\documentclass[11pt, a4paper]{article}
\usepackage{t1enc}
\usepackage[latin1]{inputenc}
\usepackage[english]{babel}
\usepackage{amsmath,amsthm}
\usepackage{amsfonts}
\usepackage{latexsym}
\usepackage[dvips]{graphicx}
\usepackage{float}
\usepackage[natural]{xcolor}
\usepackage{algorithm}
\usepackage{algorithmic}
\usepackage{enumerate,enumitem}
\usepackage{multirow}
\usepackage{xcolor}
\usepackage[colorlinks,linkcolor=blue]{hyperref}
\usepackage{lineno}
\usepackage{setspace}
\usepackage{fullpage}
\usepackage[title]{appendix}
\usepackage{todonotes}
\usepackage{amssymb}

\usepackage{listings}
\usepackage{bbm}

\usepackage{fullpage}
\usepackage{hyperref}
\usepackage[margin=2.3cm]{geometry}
\usepackage{enumitem,verbatim}

\newtheorem{theorem}{Theorem}[section]
\newtheorem{claim}[theorem]{Claim}
\newtheorem{lemma}[theorem]{Lemma}
\newtheorem{corollary}[theorem]{Corollary}
\newtheorem{proposition}[theorem]{Proposition}

\newcounter{propcounter}

\theoremstyle{definition}
\newtheorem{problem}{Problem}

\newtheorem{definition}[theorem]{Definition}

\newcommand{\<}{\subseteq}

\title{\LARGE Extremal density for subdivisions with length or sparsity constraints}

\author{
Jaehoon Kim \thanks{Department of Mathematical Science, KAIST, South Korea. Email: {\tt jaehoon.kim@kaist.ac.kr}. Supported by the Fulbright Visiting Scholar Fellowship and by the National Research Foundation of Korea (NRF) grant funded by the Korean government(MSIT) No. RS-2023-00210430.
}
\and
Hong Liu \thanks{Extremal Combinatorics and Probability Group (ECOPRO), Institute for Basic Science (IBS), Daejeon, South Korea. Email: {\tt hongliu@ibs.re.kr}. Supported by IBS-R029-C4.}
\and
Yantao Tang \thanks{Zhongtai Securities Institute for Financial Studies, Shandong University, Jinan, Shandong, China. Email: {\tt yttang@mail.sdu.edu.cn}}
\and
Guanghui Wang \thanks{School of Mathematics, Shandong University, Jinan, Shandong, China. Email: {\tt \{ghwang,dlyang\}@sdu.edu.cn}. GW is supported by National Key R\&D Program of China (2023YFA1009603), by Young Taishan Scholar program of Shandong Province (201909001) and by Natural Science Foundation of China (11871311). DY is supported by Natural Science Foundation of China (12101365) and by Natural Science Foundation of Shandong Province (ZR2021QA029).}
\and
Donglei Yang\footnotemark[4]
\and
Fan Yang \thanks{Data Science Institute, Shandong University, Jinan, Shandong, China, and Extremal Combinatorics and Prob- ability Group (ECOPRO), Institute for Basic Science (IBS), Daejeon, South Korea. Email: {\tt fyang@sdu.edu.cn}. Supported by Natural Science Foundation of China (12301447), by Natural Science Foundation of Shandong Province (ZR2024QA056), by the China Postdoctoral Science Foundation (12570073310023) and by China Scholarship Council and IBS-R029-C4.}
\and
}
\date{}

\begin{document}
\maketitle
\begin{abstract}
Given a graph $H$, a balanced subdivision of $H$ is obtained by replacing all edges of $H$ with internally disjoint paths of the same length. In this paper, we prove that for any graph $H$, a linear-in-$e(H)$ bound on average degree guarantees a balanced $H$-subdivision. This strengthens an old result of Bollob\'as and Thomason, and resolves a question of Gil-Fern\'{a}ndez, Hyde, Liu, Pikhurko and Wu.

We observe that this linear bound on average degree is best possible whenever $H$ is logarithmically dense. We further show that this logarithmic density is the critical threshold: for many graphs $H$ below this density, its subdivisions are forcible by a sublinear-in-$e(H)$ bound on average degree. We provide such examples by proving that the subdivisions of any almost bipartite graph $H$ with sublogarithmic density
are forcible by a sublinear-in-$e(H)$ bound on average degree, provided that $H$ satisfies some additional separability condition.
\end{abstract}

\section{Introduction}
For a graph $H$, a \emph{subdivision} of $H$, denoted by $TH$, is a graph obtained by replacing edges of $H$ by internally vertex-disjoint paths. This is a fundamental concept for studying topological and structural aspects of graphs as
a subdivision of $H$ has the same topological structure as $H$. For example, the celebrated theorem of Kuratowski \cite{Kura} in 1930 used this notion to characterize the planar graphs, proving that a graph is planar if and only if it contains no subdivision of $K_5$ or $K_{3,3}$.

A well-studied direction of research is to find sufficient conditions on a graph $G$ that would guarantee the existence of an $H$-subdivision in $G$. For instance, Haj\'{o}s \cite{Hajconj} conjectured in 1961 a strengthening of Hadwiger's conjecture that every graph $G$ with chromatic number $\chi(G)\geq t$ contains a $TK_t$. Dirac \cite{sub4} showed that this conjecture is true for $t\leq 4$, but in 1979 Catlin \cite{Hajos} disproved the conjecture for all $t\geq 7$. Later, Erd\H{o}s and Fajtlowicz \cite{Erd} showed that the conjecture is false for almost all graphs by considering random graphs, see also~\cite{largegirth1,largegirth2} for more recent developments.
As a stronger and more fundamental question, conditions on average degree guaranteeing an $H$-subdivision have been extensively studied, starting from a result of Mader \cite{ma67} from 1967. Mader showed that sufficiently large constant average degree implies a large clique subdivision. More precisely, for every $k\in \mathbb{N}$, there exists (finite) $f(k)$ such that every graph $G$ with average degree at least $f(k)$ contains a $TK_k$. Mader furthermore conjectured that one can take $f(k)=O(k^2)$. This conjecture was finally resolved in the 90s by Bollob\'{a}s and Thomason \cite{BolTto} and independently by Koml\'{o}s and Szemer\'{e}di \cite{KS}. Jung \cite{Jung}  observed that $K_{k^2/10,k^2/10}$ does not contain $TK_k$, hence the quadratic bound on $f(k)$ is optimal.

As this example implies that the quadratic bound is best possible, the following stability-type question naturally arises. Can we find a larger clique subdivision in $G$ if it does not structurally look like the (disjoint union of) dense bipartite graphs?
One way to formalize this question was suggested by Mader \cite{C4free}, conjecturing that 
every $C_4$-free graph $G$ with average degree $d(G)=\Omega(k)$ contains a $TK_{k}$. After some partial results (see e.g.~\cite{C6free,largegirth1,largegirth2}), this conjecture was resolved by Liu and Montgomery \cite{LiuC4}. In fact, they proved a stronger statement that for every $t\geq s\geq 2$, there exists a constant $c=c(s,t)$ such that if $G$ is $K_{s,t}$-free with $d(G)=d$, then $G$ contains a $TK_{cd^{\frac{s}{2(s-1)}}}$.

Another way to formalize the question was suggested by Liu and Montgomery \cite{LiuC4}.
Observing that the disjoint union of dense bipartite graphs has a small-sized subgraph with almost the same average degree, Liu and Montgomery conjectured that if a graph $G$ has $\omega(k^2)$ vertices and has no small induced subgraphs with almost the same average degree as the entire graph, then there exists $h(k)=o(k^2)$ such that average degree at least $h(k)$ yields a $TK_{k}$. This conjecture was resolved by Im, Kim, Kim and Liu \cite{crux} using the notion of `crux'.

\subsection{Balanced subdivisions}
Recent trends have focused on the existence of subdivisions with length constraints. In particular, a subdivision of $H$ is \emph{balanced} if every edge of $H$ is subdivided the same number of times. For $\ell\in \mathbb{N}$, denote by $TH^{(\ell)}$ a balanced subdivision of $H$ where every edge of $H$ is subdivided $\ell$ times. For dense graphs, an old conjecture of Erd\H{o}s \cite{Erd1} states that for every $\varepsilon>0$, there exists $\delta>0$ such that every graph with $n$ vertices and at least $\varepsilon n^2$ edges contains a $TK_{\delta \sqrt{n}}^{(1)}$. Alon, Krivelevich and Sudakov \cite{Alon} confirmed the conjecture with $\delta=\varepsilon^{\frac{3}{2}}$, and this result was improved to $\delta=\varepsilon$ by Fox and Sudakov \cite{depran}. In the sparse regime, Thomassen \cite{Thom1} in the 80s conjectured a strengthening of Mader's result~\cite{ma67} that large constant average degree suffices to force a large balanced clique subdivision: for each $k\in \mathbb{N}$, there exists some $g(k)$ such that every graph $G$ with $d(G)\geq g(k)$ contains a $TK^{(\ell)}_k$ for some $\ell\in \mathbb{N}$. Very recently, Thomassen's conjecture was resolved in the positive by Liu and Montgomery \cite{Adj}. Wang \cite{wang} later gave a quantitative improvement, showing that one can take $g(k)=k^{2+o(1)}$. Finally, the optimal quadratic bound $g(k)=O(k^2)$ was given by Luan, Tang, Wang and Yang \cite{LTWY} and independently by Gil-Fern\'{a}ndez, Hyde, Liu, Pikhurko, and Wu \cite{ques}. In \cite{LTWY}, the result of \cite{LiuC4} was also strengthened to a balanced version, i.e. every $C_4$-free graph contains a balanced clique subdivision of order linear in its average degree.

In this paper, we focus on forcing $H$-subdivisions for general graphs $H$. Bollob\'as and Thomason~\cite{Highlink} proved a nice structural result that highly connected graphs are highly linked. Their result, together with Mader's result \cite{Mader0} on subgraphs with high connectivity, implies that for any graph $H$ with no isolated vertices, every graph with average degree at least $100e(H)$ contains a subdivision of $H$. Note that when $H$ is a clique, the linear-in-$e(H)$ bound recovers the quadratic bound in \cite{BolTto,KS}. However, the structural linkage approach in \cite{Highlink} fails to provide any control on how edges in $H$ are subdivided. Gil-Fern\'{a}ndez, Hyde, Liu, Pikhurko and Wu \cite{ques} raised the problem of whether the same linear bound $O(e(H))$ suffices to force a balanced $H$-subdivision.

\begin{problem}[\cite{ques}]\label{problem:H}
    Does there exist a constant $C$ such that, for any graph $H$ with no isolated vertices, any graph $G$ with average degree at least $C\cdot e(H)$ contains a balanced subdivision of $H$?
\end{problem}

Our first result answers Problem~\ref{problem:H} in the affirmative.
\begin{theorem}\label{main1}
There exists a constant $C>0$ such that, for any graph $H$ with no isolated vertices, if $G$ is a graph with average degree $d(G)\geq C\cdot e(H)$, then $G$ contains a $TH^{(\ell)}$ for some $\ell\in \mathbb{N}$.
\end{theorem}

\subsection{When does a sublinear bound suffice?}\label{sec:1.2}
Notice that an observation of Jung \cite{Jung} shows that the linear-in-$e(H)$ bound is optimal when $H$ is a clique. It is a natural problem to study when a sublinear bound suffices to ensure an $H$-subdivision. A specific question of this sort was proposed by Wood in the Barbados workshop (2020).

\begin{problem}[Wood]\label{ques-Ktt}
	For given $k\in \mathbb{N}$, does there exist $h(k)= o(k^2)$, such that every graph with average degree at least $h(k)$ contains a subdivision of $K_{k,k}$?
	\end{problem}

This problem essentially asks whether the structure of $H$ could affect the density needed to force an $H$-subdivision. To understand why the above question imposes a bipartite condition among many other structural conditions, consider the following example.
For $0< \delta < 1$, consider a graph $H$ with $(1+\delta)m$ edges having no spanning bipartite subgraph with more than $m$ edges (e.g. certain blowups of a triangle). Note that the graph $G=K_{\delta m/2, \delta m/2}$ does not contain any $H$-subdivision while $d(G)\geq \frac{1}{5} \delta e(H)$. Indeed, any attempt to embed a copy of $TH$ in $G$ would require replacing at least $\delta m$ edges of $H$ with paths in $G$, which occupy at least $\delta m\ge |G|$ additional vertices in $G$, a contradiction. 
Hence, in order for a sublinear bound $d(G)\geq \varepsilon e(H)$ to guarantee an $H$-subdivision, the graph $H$ must be almost bipartite, in the sense that deleting $O(\varepsilon e(H))$ edges from $H$ leaves a bipartite graph.


However, Im, Kim, Kim and Liu \cite{crux} recently observed that bipartiteness on $H$ alone is not sufficient, so
the answer to Problem~\ref{ques-Ktt} is no. They showed that regardless of the structure of $H$, the linear bound $O(e(H))$ cannot be improved as long as $H$ is dense, i.e.~when $d(H)=\Omega(|H|)$. We prove that a more careful analysis of their construction shows that for any logarithmically dense $H$, the linear bound is optimal.

\begin{proposition}\label{random}
For any $h$-vertex graph $H$ with $d(H)\geq 128\log h$, there exists an $n$-vertex graph $G$ for all sufficiently large $n$ such that $d(G)\geq \frac{e(H)}{40}$ and $TH\nsubseteq G$.
\end{proposition}

Thus, to search for graphs $H$ for which a $o(e(H))$-bound on average degree suffices to force an $H$-subdivision, one has to look into those sparser almost-bipartite graphs with $d(H)=o(\log h)$. With this proposition, the following natural question arises.
Here, we say that $H$ is $\alpha$-almost-bipartite if one can delete $\alpha e(H)$ edges to make $H$ bipartite.

\begin{problem}\label{natural}
For given $\varepsilon>0$, do there exist $\alpha,c,K, h_0$ satisfying the following for all $h\geq h_0$?
For any $h$-vertex $\alpha$-almost-bipartite graph $H$ with $K\leq d(H),\Delta(H)\leq c \log h$,  every graph $G$ with average degree at least $\varepsilon e(H)$ contains a subdivision of $H$.
	\end{problem}
Here, the condition $d(H)\geq K$ is imposed merely to avoid some trivial counterexamples such as graphs $H$ having more than $\varepsilon e(H)$ vertices.
Indeed, our next theorem proves that the answer to this problem is yes if we impose an additional separability condition. To ease the notation, we give the following notion of biseparability, which incorporates both almost-bipartiteness and separability.

\begin{definition}[Biseparable]\label{def: separable}
A graph $H$ is called \emph{$(s,k)$-biseparable} if there exists $E_1\subseteq E(H)$ with $|E_1|\leq s$ such that $H\backslash E_1$ is bipartite and every component in $H\backslash E_1$ has at most $k$ vertices.
\end{definition}

\begin{theorem}\label{thm: sublog}
For given $\varepsilon>0$, there exist $\alpha,c, K>0$ and $h_0$ satisfying the following for all $h>h_0$. If $H$ is an $h$-vertex $(\alpha e(H), c \log h)$-biseparable graph with $d(H)\geq K$, then any graph $G$ with $d(G)\geq \varepsilon e(H)$ contains a $TH$.
\end{theorem}

This theorem also shows that the logarithmic density of $H$ in Proposition~\ref{random} is the correct threshold above which the linear-in-$e(H)$ bound forcing an $H$-subdivision is necessary.

The conditions in Theorem~\ref{thm: sublog} might seem technical. However, Theorem~\ref{thm: sublog} covers many well-studied families of graphs. For instance, it includes almost-bipartite graphs with bounded maximum degree
from any proper minor-closed classes~\cite{HNW1, HNW2} and classes of graphs with polynomial expansion~\cite{S3}. Note that the maximum degree condition is needed only to ensure the edge-separability for the graphs in those classes.

Below, we provide two further interesting families of graphs $H$ for which a sublinear bound suffices to force an $H$-subdivision. 
The first family demonstrates that Theorem~\ref{thm: sublog} applies to a class of graphs that are not only practically useful but also relevant in real-world applications.
The second family of Cartesian products is not covered by Theorem~\ref{thm: sublog}, as their separability is significantly weaker. We present this family to highlight that there are additional desirable graphs beyond those encompassed by Theorem~\ref{thm: sublog}

\smallskip

\noindent\textbf{1.~Graphs from stochastic block model.} The \emph{stochastic block model} is a model for random graphs, introduced in 1983 to study communities in social networks by Holland, Laskey and Leinhardt~\cite{HLL}. This model is well studied due to its important role in recognizing community structure in graph data in statistics, machine learning, and network science. We will work with the following bipartite version. Let $t,k,n\in\mathbb{N}$ with $n=2kt$ and $p,q\in[0,1]$. Denote by $\mathbb{G}(n,p;t,q)$ the $n$-vertex random graph with an equipartition of the vertex set $V=V_1\cup\ldots\cup V_k$ such that (1) for each $i\in[k]$, the subgraph induced on each $V_i$ is distributed as a bipartite Erd\H{o}s-Renyi random graph $G(t,t,q)$,\footnote{That is the random subgraph of $K_{t,t}$ where each edge is retained with probability $q$ independent of others.} and (2) for every distinct $i,j\in [k]$, the bipartite subgraph induced on $[V_i,V_j]$ is distributed as $G(2t,2t,p)$. Such a model is called \emph{assortative} if $q>p$.

The first family of graphs $H$ comes from (strongly assortative bipartite) stochastic block model with logarithmic community sizes.

\begin{corollary}\label{thre-pretty-version}
There exists a universal $c>0$ such that the following holds. For any $\varepsilon>0$, there are $\delta>0$ and $h_0\in\mathbb{N}$ satisfying the following. Let $h\ge h_0,~t=c\log h$ and $0<p<\frac{\delta \log h}{h}$. If $H\sim \mathbb{G}(h,p;t,\frac{1}{2})$, then with probability $1-o_h(1)$ every $n$-vertex graph $G$ with $n>h$ and $d(G)\geq \varepsilon e(H)$ contains a $TH$.
\end{corollary}

We remark that with high probability, a graph $H\sim \mathbb{G}(h,p;t,\frac{1}{2})$ above has logarithmic density: $d(H)= \Omega(t)=\Omega(\log h)$.
Furthermore, standard concentration inequalities yield that  with high probability $H$ is $(o(e(H)),o(\log h))$-biseparable, and then  Theorem~\ref{thm: sublog} immediately implies
Corollary~\ref{thre-pretty-version}.
\medskip

\noindent\textbf{2.~Cartesian products of bipartite planar graphs.} Given two graphs $G$ and $H$, the \emph{Cartesian product} of $G$ and $H$, denoted by $G\Box H$, is the graph with vertex set $V(G)\times V(H)$ such that two vertices  $(x,y)$ and $(x',y')$ are adjacent if and only if $(i)$ $x=x'$
and $yy'\in E(H)$, or $(ii)$ $y=y'$ and $xx'\in E(G)$. Denote by $G^{\Box r}$ the Cartesian product of $r$ copies of $G$. It is known that (\cite{carproduct}) the Cartesian product of bipartite graphs is also bipartite.

The second family consists of Cartesian products of bounded degree planar graphs.

\begin{theorem}\label{ans2-pretty-version}
For any $\varepsilon>0$ and $D\in\mathbb{N}$, there exists $K>0$ such that the following holds for all $f\ge K$. Let $F$ be an $f$-vertex bipartite planar graph with $1\leq d(F), \Delta(F)\le D$ and let $H=F^{\Box r}$. If $r\le\frac{\log\log f}{K}$, then any $n$-vertex graph $G$ with $n\ge f^r$ and $d(G)\geq \varepsilon e(H)$ contains a $TH$.
\end{theorem}

Indeed, the above condition $\Delta(F)\le D$ can be replaced with $\Delta(F)= o(\sqrt{\log f})$ without much changes in our argument (see Lemma~\ref{partition} and the paragraph afterwards for an elaborate discussion).

\subsection{Related work}
The analogous problem of determining the extremal density required to force a minor of a given graph $H$ has also attracted considerable attention. The extremal function here, denoted by $c(H)$, is the supremum of average degrees of graphs not containing $H$ as a minor. One classical such result is by Kostochka~\cite{HNW017} and independently Thomason~\cite{HNW033} that $c(K_k)=\Theta(k\sqrt{\log k})$. Later, Thomason and Wales~\cite{HNW036} showed that for general graphs $H$, $c(H)=O(|V(H)|\sqrt{d(H)})$, which is optimal for almost all polynomially dense $H$. Analogous to Problem~\ref{ques-Ktt}, finding graphs $H$ with $c(H)=o(|V(H)|\sqrt{d(H)})$ has gained much attention. Here are some families of such $H$: complete bipartite graphs $K_{s,t}$~\cite{HNW018, HNW020}, disjoint unions of cycles~\cite{HNW03} and graphs with strong separation properties \cite{Norin}. In particular, Hendrey, Norin and Wood \cite{Norin} proved that (among others) the hypercube $Q_d$ has $c(Q_d)=O(|V(Q_d)|)$. Note that Theorem~\ref{ans2-pretty-version} does not apply to the hypercube. It would be interesting to know whether a sublinear bound suffices to force a subdivision of $Q_d$.

\begin{problem}
	For given $d\in \mathbb{N}$, does there exist $q(d)= o(d2^d)$, such that every graph with average degree at least $q(d)$ contains a subdivision of $Q_d$?
	\end{problem}

\noindent
\textbf{Organization.} The rest of the paper is organized as follows. Preliminaries are given in Section \ref{Pre2} before the proof of Proposition \ref{random} in Section \ref{proproof}. In Section \ref{pf1}, we give overviews of the proof strategies and pack the main steps of Theorems \ref{main1}, \ref{thm: sublog} and \ref{ans2-pretty-version} into Lemmas \ref{dense}, \ref{dense2-3} and \ref{inte}, respectively; the proofs of these three main lemmas are given in Sections~\ref{Thm3simi}, \ref{sec:log-dense} and \ref{sec:log-sparse}, respectively.

\section{Preliminaries}\label{Pre2}
\subsection{Notation}
Denote by $X\sim \mathrm{Bin}(n,p)$ the random variable drawn according to the binomial distribution with parameters $n$ and $p$. For any positive integer $r$, we write $[r]$ for the set $\{1, \ldots, r\}$.
Given a graph $G=(V,E)$, we denote by $d(G)$ and $\delta(G)$ the average degree and the minimum degree of $G$, respectively. Given a set $W\subseteq V(G)$, denote its \emph{external neighbourhood} by $N_G(W):=\{u\notin W: uv\in E(G) \ \text{for} \ \text{some} \ v\in W\}$. Furthermore, set $N_G^{0}(W):=W$ and $N_G^1(W):=N_G(W)$ and for each $i\geq 1$, define $N_G^{i+1}(W):=N(N_G^{i}(W))\setminus N_G^{(i-1)}(W)$. Denote by $B_G^r(W)$ the ball of radius $r$ around $W$, that is, $B_G^r(W)=\cup_{i\leq r}N_G^{i}(W)$. For simplicity, write $B_G^r(v)$ for $B_G^r(\{v\})$. For any set $W\subset V(G)$, the subgraph of $G$ induced on $W$, denoted as $G[W]$, is the graph with vertex set $W$ and edge set $\{xy\in E(G)|\ x,y\in W\}$, and write $G-W=G[V(G)\backslash W]$.
For any $A,B\subseteq V(G)$, we denote by $G[A,B]$ the graph with vertex set $A\cup B$ and edge set $\{xy\in E(G)|\ x\in A, y\in B\}$. We simply use $e_G(A,B)=|E(G[A,B])|$. Moreover, we define the \emph{density} between $A$ and $B$ to be
\begin{equation*}
d_G(A,B)=\frac{e_G(A,B)}{|A||B|}.
\end{equation*}

For a path $P$, the length of $P$ is the number of edges in $P$, and we say $P$ is an $x,y$-path if $x$ and $y$ are the endvertices of $P$. Given a family of graphs $\mathcal{F}$, denote by $|\mathcal{F}|$ the number of graphs in $\mathcal{F}$ and we write $V(\mathcal{F})=\cup_{G\in \mathcal{F}}V(G)$. A graph $G$ is said to be \emph{$k$-degenerate} if every nonempty subgraph $H$ of $G$ has a vertex of degree at most $k$ in $H$.

Throughout the paper, we omit floor and ceiling signs when they are not essential. Also, we use standard hierarchy notation, that is, we write $a\ll b$ to denote that given $b$ one can choose $a_0$ such that the subsequent arguments hold for all $0<a\le a_0$.
\subsection{Tools}
Let $\mathrm{ex}(n,H)$ be the maximum number of edges in an $H$-free graph on $n$ vertices. The classic lemma in the following gives an upper bound for $\mathrm{ex}(n,K_{s,t})$.
\begin{lemma}[\cite{KST}, \rm{K\H{o}v\'{a}ri-S\'{o}s-Tur\'{a}n}]\label{exKST}
For all integers $1\leq s\leq t$, $\mathrm{ex}(n,K_{s,t})\leq t^{\frac{1}{s}}n^{2-\frac{1}{s}}$.
\end{lemma}
The following result of Lee \cite{Lee} will be useful for embedding large bipartite graphs with a bounded degeneracy.
\begin{lemma}[\cite{Lee}]\label{Leethm}
There exists $K>0$ such that the following holds for every natural number $\kappa$ and real number $\alpha\leq\frac{1}{2}$. For every natural number $n\geq\alpha^{-K\kappa^2}$, if $G$ is a graph with at least $\alpha^{-K\kappa}n$ vertices and density at least $\alpha$, then it contains all graphs in the family of $\kappa$-degenerate bipartite graphs on $n$ vertices as subgraphs.
\end{lemma}
We discuss the regularity lemma that will be used for embedding certain subgraphs. Firstly, we introduce the following two definitions.
\begin{definition}[$\varepsilon$-regular pair]
Let $G$ be a graph and $X,Y\subseteq V(G)$. We call $(X,Y)$ an \emph{$\varepsilon$-regular pair} (in $G$) if for all $A\subset X, B\subset Y$ with $|A|\geq \varepsilon|X|, |B|\geq \varepsilon|Y|$, one has
\begin{equation*}
|d(A,B)-d(X,Y)|\leq \varepsilon.
\end{equation*}
Additionally, we say that $(X,Y)$ is \emph{$(\varepsilon,\beta)$-regular} if $d(X,Y)\geq \beta$ for some $\beta>0$.
\end{definition}
\begin{lemma}[\cite{RD}]\label{regular}
Let $(A,B)$ be an $(\varepsilon,\beta)$-regular pair, and let $Y\subseteq B$ have size $|Y|\geq \varepsilon|B|$. Then all but fewer than $\varepsilon|A|$ of the vertices in $A$ have $($each$)$ at least $(\beta-\varepsilon)|Y|$ neighbors in $Y$.
\end{lemma}

\begin{definition}[Regular partition]
For $r\in \mathbb{N}$, a partition $\mathcal{P}=\{V_0,V_1,\ldots,V_r\}$ of $V(G)$ is \emph{$\varepsilon$-regular} if
\begin{itemize}
\item[$(i)$] $|V_0|\leq \varepsilon|V(G)|$,
\item[$(ii)$] $|V_1|=|V_2|=\cdots=|V_r|$,
\item[$(iii)$] all but $\varepsilon r^2$ pairs $(V_i,V_j)$ with $1\leq i<j\leq r$ are $\varepsilon$-regular.
\end{itemize}
\end{definition}

We need the following form of the regularity lemma.
\begin{lemma}[\cite{SzeRL}, Szemer\'{e}di's regularity lemma]\label{regudegver}
For every $\varepsilon>0$, there exists $M=M(\varepsilon)>0$ such that for any graph $G=(V,E)$ and $\beta\in[0,1]$, there is an $\varepsilon$-regular partition $\mathcal{P}=\{V_0,V_1,\ldots,V_r\}$ of $V$ for some $r\in \mathbb{N}$ and a subgraph $G'=(V,E')$ with the following properties:
\begin{itemize}
\item[$(1)$] $1/\varepsilon\le r\leq M$,
\item[$(2)$] $|V_i|\leq \varepsilon|V|$ for all $i\geq 1$,
\item[$(3)$] $d_{G'}(v)>d_{G}(v)-(\beta+\varepsilon)|V|$ for all $v\in V$,
\item[$(4)$] $e(G'[V_i])=0$ for all $i\geq 1$,
\item[$(5)$] every pair $G'(V_i,V_j)$, $1\leq i<j\leq r$, is $\varepsilon$-regular, with density either $0$ or greater than $\beta$.
\end{itemize}
\end{lemma}

\subsection{Sublinear expander}
Koml\'{o}s and Szemer\'{e}di \cite{KS1,KS} introduced a notion of sublinear expander, that is, a family of graphs in which any subset of vertices of reasonable size expands by a sublinear factor. For $\varepsilon_1>0$ and $k>0$, we first set the function $\rho(x)$ 
\begin{align*}
\begin{split}
\rho(x)=\rho(x,\varepsilon_1,k):=\left\{
 \begin{array}{ll}
  0    & \text{if} \ x<\frac{k}{5}, \\
  \frac{\varepsilon_1}{\log^2(\frac{15x}{k})}   & \text{if} \ x\geq \frac{k}{5}.
 \end{array}
 \right.
 \end{split}
\end{align*}

\noindent
For simplicity, we write $\rho(x)$ for $\rho(x,\varepsilon_1,k)$ when $\varepsilon_1$ and $k$ are clear from context. Note that $\rho(x)$ is a decreasing function when $x\geq \frac{k}{5}$.

\begin{definition}[\cite{KS1,KS}]
Let $\varepsilon_1>0$ and $k>0$. A graph $G$ is an \emph{$(\varepsilon_1,k)$-expander} if
\begin{equation*}
|N(X)|\geq \rho(|X|)\cdot|X|
\end{equation*}
for all $X\subseteq V(G)$ of size $\frac{k}{2}\leq |X|\leq \frac{|V(G)|}{2}$.
\end{definition}

Koml\'{o}s and Szemer\'{e}di \cite{KS} showed that every graph $G$ contains a sublinear expander almost as dense as $G$.
\begin{lemma}[\cite{KS}]\label{ave2}
There exist $0<\varepsilon_0, \varepsilon_1<\frac{1}{8}$ such that for any $k\in \mathbb{N}$ every graph $G$ contains an $(\varepsilon_1,k)$-expander $H(V,E)$ with
\begin{equation*}
d(H)\geq \frac{d(G)}{1+\varepsilon_0}\geq\frac{d(G)}{2} \ \text{and} \ \delta(H)\geq \frac{d(H)}{2},
\end{equation*}
which has the following additional robust property where $n=|V|$. For every $X\subseteq V$ with $|X|<\frac{n\rho(n)d(H)}{4\Delta(H)}$, there is a subset $Y\subseteq V\backslash X$ of size $|Y|>n-\frac{2\Delta(H)}{d(H)}\cdot\frac{|X|}{\rho(n)}$ such that $H[Y]$ is still an $(\varepsilon_1,k)$-expander. Moreover, $d(H[Y])\geq \frac{d(H)}{2}$.
\end{lemma}
The `moreover' part can be easily obtained by going through their proof in \cite{KS}, though it is not explicitly stated in the original lemma.

\begin{proposition}\label{bipsub}
Every graph $G$ contains a bipartite subgraph $H$ with $d(H)\geq \frac{d(G)}{2}$.
\end{proposition}
Combining Proposition \ref{bipsub} with Lemma \ref{ave2}, we immediately obtain the following corollary.
\begin{corollary}\label{expander}
There exists $\varepsilon_1>0$ such that the following holds for every $k>0$ and $d\in \mathbb{N}$. Every graph $G$ with $d(G)\geq 8d$ has a bipartite $(\varepsilon_1,k)$-expander $H$ with $\delta(H)\geq d$.
\end{corollary}

\begin{proposition}\label{proball}
Let $m$ be the smallest even integer which is larger than $\log^4\frac{n}{d}$. If $G$ is an $(\varepsilon_1,\varepsilon_2d)$-expander and $v\in V(G)$ has $d(v)\geq \varepsilon_2d$, then $|B_G^m(v)|\geq \frac{n}{2}$.
\end{proposition}

Proposition \ref{proball} tells that any vertex of large degree can expand into a very large ball of radius at most $m$, and we postpone its proof to Appendix~\ref{app1}. A key property of expanders that we shall use is to connect vertex sets with a short path whilst avoiding a reasonable-sized set of vertices.
\begin{lemma}[\cite{KS}]\label{distance}
Let $\varepsilon_1, k>0$. If $G$ is an $n$-vertex $(\varepsilon_1,k)$-expander, then for any two vertex sets $X_1, X_2$ each of size at least $x\geq k$, and a vertex set $W$ of size at most $\frac{\rho(x)x}{4}$, there exists a path in $G-W$ between $X_1$ and $X_2$ of length at most $\frac{2}{\varepsilon_1}\log^3\left(\frac{15n}{k}\right)$.
\end{lemma}

\subsection{Proof of Proposition \ref{random}}\label{proproof}
We recall the following Chernoff's inequality (see \cite{refcher}, Theorem 2.1). Let $X$ be any sum of independent (not necessarily identical) Bernoulli random variables. For $0<\eta<1$, it holds that 
\begin{equation}\label{chebound}
    \mathbb{P}[X<(1-\eta)\mathbb{E}X]<e^{-\tfrac{\eta^2\mathbb{E}X}{2}}.
\end{equation}

We first consider the case when $n=\frac{e(H)}{5}$. For larger values of $n$, one can take a disjoint union of $G$ given as below. Let $G=G(A,B,\frac{1}{2})$ be an $n$-vertex random bipartite graph, where $|A|=|B|:=n_1=\frac{e(H)}{10}$. We shall verify that with positive probability,  $d(G)\geq \frac{e(H)}{40}$ and $TH\nsubseteq G$.
First, let $X_1$ denote the number of edges in $G(A,B,\frac{1}{2})$, then by inequality (\ref{chebound}),
\begin{equation}\label{14ineq}
\mathbb{P}\left[X_1\leq \frac{\mathbb{E}[X_1]}{2}\right]< e^{-\frac{\mathbb{E}[X_1]}{8}}.
\end{equation}

Let $\mathcal{F}$ denote the set of all injections from $V(H)$ to $V(G)$. Note that $|\mathcal{F}|\leq n^h$. To find in $G$ a subdivision of $H$, we first fix an injection $f\in \mathcal{F}$ and if an edge $uv\in E(H)$ satisfies $f(u)f(v)\notin E(G)$ (we call it \emph{missing} in $G$), then we need a path of length at least $2$ in $G$ to connect $f(u)$ and $f(v)$. Moreover, all such paths are internally vertex disjoint. Thus, if the number of missing edges is at least $\frac{e(H)}{4}$,  then we can not find a $TH$ in $G$ under the injection $f$ since each missing edge requires a distinct internal vertex in $G$ and so $|V(TH)|\geq \frac{e(H)}{4}>2n_1=n$.

Hence our task is to find a graph $G$ in which every $f\in \mathcal{F}$ witnesses many missing edges. For a fixed $f\in \mathcal{F}$, let $X_f$ be the random variable counting the missing edges under $f$ in $G(A,B,\frac{1}{2})$. Let $M(f)$ be the set of edges $e=uv$ in $H$ such that $f(u)$ and $f(v)$ lie in the same part, and $B(f)=E(H)\backslash M(f)$. Moreover, write $m(f)=|M(f)|$. Let $Y_f$ be a random variable such that $Y_f\sim \mathrm{Bin}(|B(f)|,\frac{1}{2})$. Then we have $X_f=m(f)+Y_f$, and
\begin{equation*}
\mathbb{E}[X_f]=m(f)+\frac{e(H)-m(f)}{2}\geq \frac{e(H)}{2}.
\end{equation*}
Then it suffices to consider the case $m(f)<\tfrac{e(H)}{4}$ and as $\mathbb{E}[Y_f]=\tfrac{1}{2}|B(f)|>\tfrac{3e(H)}{8}$, we have
\begin{align}
\mathbb{P}\left[X_f\leq \frac{e(H)}{4}\right]&\leq \mathbb{P}\left[X_f\leq \mathbb{E}[X_f]-\frac{e(H)}{4}\right] = \mathbb{P}\left[m(f)+Y_f\leq m(f)+\mathbb{E}[Y_f]-\frac{e(H)}{4} \right] \nonumber \\
&=\mathbb{P}\left[Y_f\leq \mathbb{E}[Y_f]-\frac{e(H)}{4}\right]\leq e^{-\frac{e(H)^2}{32\mathbb{E}[Y_f]}}\leq e^{-\frac{e(H)}{32}},  \nonumber
\end{align}
where the penultimate inequality holds by inequality (\ref{chebound}). By the union bound, recalling that $n=\frac{e(H)}{5}$, we have
\begin{align}
\mathbb{P}\left[\bigcap_{f\in \mathcal{F}}\left(X_f\geq  \tfrac{e(H)}{4}\right)\right]\geq 1-n^{h}e^{-\frac{e(H)}{32}}=1-e^{h\log(\frac{e(H)}{5})-\frac{e(H)}{32}}>\frac{1}{2}, \nonumber
\end{align}
where the last inequality holds as $e(H)\leq h^2$ and $d(H)\geq 128\log h$.
Hence, together with (\ref{14ineq}), we have that there exists a bipartite graph $G$ such that
$e(G)\geq \frac{n_1^2}{4}$ and for every $f\in \mathcal{F}$, the number of missing edges in $G$ is at least $\tfrac{e(H)}{4}$ under the injection $f$. So $d(G)\ge \tfrac{n_1}{4}= \tfrac{e(H)}{40}$ and $TH\nsubseteq G$ as desired.

\section{Main lemmas and overviews}\label{pf1}
\subsection{Proof of Theorem \ref{main1}}
Note that by Corollary \ref{expander}, $G$ contains a bipartite sublinear expander as a subgraph. We divide the proof of Theorem \ref{main1} into two cases depending on whether the bipartite subgraph is dense or sparse. The sparse case (see Lemma \ref{sparse}) follows from a recent result of Wang \cite{wang} on balanced clique subdivisions. The dense case is the most involved task and the bulk of
the work is to handle dense expanders (see Lemma \ref{dense}). Throughout the proof we always assume $H$ is a graph without isolated vertices.

\begin{lemma}\label{dense}
Suppose $\frac{1}{n}, \frac{1}{d}\ll\frac{1}{C}\ll \varepsilon_1,\varepsilon_2<\frac{1}{5}$ and $s, q\in \mathbb{N}$ satisfy $s\geq 1600$, $\log^s n \leq d \leq n$ and $q\leq \frac{d}{C}$. If $H$ is a graph with $q$ edges and $G$ is an $n$-vertex bipartite $(\varepsilon_1,\varepsilon_2d)$-expander with $\delta(G)\geq d$, then $G$ contains a $TH^{(\ell)}$ for some $\ell\in \mathbb{N}$.
\end{lemma}
\begin{lemma}[\cite{wang}]\label{sparse}
Suppose $\frac{1}{n}, \frac{1}{d}, c \ll \varepsilon_1,\varepsilon_2<\frac{1}{5}$ and $s\in \mathbb{N}$ satisfies $s\geq 20$ and $\log^s n> d$. If $G$ is an  $n$-vertex $TK_{\frac{d}{2}}^{(2)}$-free bipartite $(\varepsilon_1,\varepsilon_2d)$-expander with $\delta(G)\geq d$, then $G$ contains a $TK_{cd}^{(\ell)}$ for some $\ell\in \mathbb{N}$.
\end{lemma}

\begin{proof}[Proof of Theorem \ref{main1}]
Take $\varepsilon_2=\frac{1}{10}, s=1600$, then we obtain a constant $\varepsilon_1$ from Corollary \ref{expander}.
We choose $C>0$ satisfying $\frac{1}{d}\ll \frac{1}{C} \ll \varepsilon_1, \varepsilon_2$.
Let $G$ be a graph with average degree $d(G)=d$  and let $H$ be a $q$-edge graph with $q< \frac{d}{C}$. Let $d_1=\frac{d}{8}$. By Corollary \ref{expander}, $G$ has a bipartite $(\varepsilon_1,\varepsilon_2d_1)$-expander $G_1$ with $\delta(G_1)\geq d_1$, and let $|V(G_1)|=n$. If $d_1\geq \log^sn$, then by Lemma \ref{dense}, $G\supseteq TH^{(\ell)}$ for some $\ell\in \mathbb{N}$. Otherwise, by Lemma \ref{sparse}, either $G\supseteq TK_{\frac{d_1}{2}}^{(2)}$ or $G\supseteq TK_{c_{\ref{sparse}}d_1}^{(\ell)}$ for some $\ell\in \mathbb{N}$. As $c_{\ref{sparse}}d_1>2q\geq |V(H)|$ and $H$ has no isolated vertices, $TH^{(\ell)}\subseteq G$ as desired.
\end{proof}

\subsection{Proof overview of Lemma \ref{dense}}
Here we give an overview of the proof of Lemma \ref{dense}. We aim to embed a balanced $TH$ for each $q$-edge graph $H$ with $q\leq \frac{d}{C}$ into the $(\varepsilon_1,\varepsilon_2d)$-expander $G$ with $\delta(G)\geq d$. Let $m=\log^4\frac{n}{d}$. If there are at least $4q$ vertices of large degree ($2dm^{12}$) in $G$, then it is easy to build a balanced $TH$ (see Lemma \ref{many}) using adjusters (see Definition~\ref{defn:adjuster}) to control lengths of paths. Otherwise we will find star-like structures that serve as the bases for building a balanced subdivision of $H$. We build for every $v\in V(H)$ a \textit{unit} or a \textit{web} (see Definitions~\ref{defn:unit} and~\ref{defn:web}) in $G$ such that all these units/webs have disjoint interiors. In order to greedily build many units and webs, we shall first prove that $G$ is locally dense (see Definition \ref{defn:locdense} and Lemma \ref{aver}). We then divide $V(H)$ into three parts $\{\mathbf{L},\mathbf{M},\mathbf{S}\}$ depending on their degrees in $H$, and equip every vertex with a unit or web accordingly (see Lemmas \ref{middlewebs} and \ref{bunit}).

Anchoring at the units or webs as above, we proceed with the connection of these units or webs in two rounds. Let $H_1$ be the spanning subgraph of $H$ with $E(H_1)$ consisting of all edges incident with vertices in $\mathbf{S}$, and $H_2=H\backslash E(H_1)$.  In the first round, we shall iteratively build, for all edges in $H_1$, internally vertex-disjoint paths in $G$ to obtain a balanced $TH_1$. The difficulty here is that the union of interiors of webs for all vertices in $\mathbf{S}$ could be relatively large ($\mathbf{S}=V(H)$ is the worst-case scenario) and we cannot hope to carry out connections completely avoiding  their interiors. To overcome this, we instead adopt an approach developed in \cite{Gps}, which we call \emph{good $\ell$-path systems}. The rough idea is that one can prepare twice as many webs as needed for vertices in $\mathbf{S}$ and discard a web once its interior is over used in the connection process. In the second round, the edges in $H_2$ are relatively easier to handle as the units/webs (for respective vertices in $\mathbf{L}$/$\mathbf{M}$) have large exteriors for robust connections.

\subsection{Proofs of Theorems \ref{thm: sublog} and \ref{ans2-pretty-version}}\label{sec:sep}



The following result on partitioning graphs with strongly sublinear separators is folklore. A \emph{balanced separator} in a graph $G$ is a set $S\subseteq V(G)$ such that every component of $G-S$ has at most $\frac{2}{3}|V(G)|$ vertices.

\begin{lemma}[\cite{lem351,lem352,lem353}]\label{partition}
Let $c>0$ and $\beta\in(0,1)$. Let $G$ be a graph with $n$ vertices such that every subgraph $G'$ has a balanced separator of size at most $c|V(G')|^{1-\beta}$. Then for all $p\ge 1$, there exists $S\subseteq V(G)$ of size at most $\frac{c2^{\beta}n}{(2^{\beta}-1)p^{\beta}}$ such that each component of $G-S$ has at most $p$ vertices.
\end{lemma}

Note that every $f$-vertex planar graph $F$ has a balanced separator of size $O(\sqrt{f})$ \cite{lem353}. Thus if $\Delta(F)=O(1)$ we may apply Lemma \ref{partition} with $p=\log f$ to obtain $S\subseteq V(F)$ of size $O(\frac{f}{\sqrt{\log f}})$ such that after removing all the $O(\frac{f}{\sqrt{\log f}})$ edges incident with $S$, each component has at most $\log f$ vertices. This immediately tells that any planar bipartite graph $F$ is $(o(f),\log f)$-biseparable provided that $\Delta(F)=O(1)$. Note that this is the only place where we need the bounded maximum degree condition to deduce Theorem~\ref{ans2-pretty-version}, and we can actually relax the condition $\Delta(F)\leq D$ to $\Delta(F)=o(\sqrt{\log f})$. As a consequence, Theorem~\ref{ans2-pretty-version} is an immediate corollary of the following theorem regarding more general graphs $F$ (not necessarily bipartite or planar) with a similar biseparability property.

\begin{lemma}\label{ans2}
Suppose $\frac{1}{f}, \frac{1}{r} \ll \frac{1}{K} , \alpha\ll \frac{1}{\kappa}, \varepsilon <1$ and $\log f > e^{K\kappa^2 r}$. If $F$ is an $f$-vertex $\kappa$-degenerate $(\alpha e(F),\log f)$-biseparable graph with $d(F)\geq 1$, and $H=F^{\Box r}$, then any graph $G$ with $d(G) \geq \varepsilon e(H)$ contains a $TH$.
\end{lemma}

Note that $|V(G^{\Box r})|=|V(G)|^r$ and $|E(G^{\Box r})|=r|V(G)|^{r-1}|E(G)|$. To see this, let $\boldsymbol{a}=(a_1,a_2,\ldots,a_r)$, $\boldsymbol{b}=(b_1,b_2,\ldots,b_r)$ be two vertices in $V(G^{\Box r})$, and recall that $\boldsymbol{a},\boldsymbol{b}$ are adjacent whenever they only differ at one coordinate and the corresponding coordinates form an edge in $G$, that is, there exists $j\in[r]$ such that $a_j\neq b_j$, $a_jb_j\in E(G)$ and $a_i=b_i$ for all $i\neq j$. 

The proofs of Theorem \ref{thm: sublog} and Lemma \ref{ans2} are split into the following two lemmas depending on the density of the host graph. Denote by $TH^{(\leq \ell)}$ the graph obtained by replacing some edges of $H$ by internally vertex-disjoint paths of length at most $\ell+1$.

\begin{lemma}[Dense case]\label{dense2-3}
Suppose $\frac{1}{h}, \frac{1}{f}, \frac{1}{r}\ll \frac{1}{K}, \alpha, c\ll \beta,\varepsilon < 1$.
\begin{enumerate}[label = (\arabic{enumi})]
\rm \item \label{deenu1} Let $H$ be an $h$-vertex $(\alpha e(H),c\log h)$-biseparable graph with $d(H)\ge K$.
Then any $n$-vertex graph $G$ with $d(G)=\beta n\geq\varepsilon e(H)$ contains a $TH^{(\leq 3)}$.
\rm \item \label{deenu2} Further suppose $\frac{1}{K},\alpha\ll \beta,\frac{1}{\kappa}$ and $\log f>e^{K\kappa^2r}$ for some $\kappa \in \mathbb{N}$. Let $F$ be an $f$-vertex $\kappa$-degenerate $(\alpha e(F),\log f)$-biseparable graph with $d(F)\ge 1$ and $H= F^{\Box r}$. Then any $n$-vertex graph $G$ with $d(G)=\beta n \geq \varepsilon e(H)$ contains a $TH^{(\leq 3)}$.
\end{enumerate}
\end{lemma}

\begin{lemma}[Sparse case]\label{inte}
Suppose $\frac{1}{h}, \frac{1}{f}, \frac{1}{r} \ll \frac{1}{K}, \alpha, c\ll\varepsilon, \varepsilon_1, \varepsilon_2, \frac{1}{\kappa}<1$ and $s\in \mathbb{N}$ satisfies $s\geq 1600$ and $\log^s n \leq d \leq \frac{n}{K}$. Let $H$ be an $h$-vertex graph satisfying arbitrary one of the following properties:
\begin{enumerate}[label = (\arabic{enumi})]
\rm \item \label{spenu1} $H$ is $(\alpha e(H),c\log h)$-biseparable with $d(H)\geq K$,
\rm \item \label{spenu2} $H=F^{\Box r}$, where $F$ is an $f$-vertex $\kappa$-degenerate $(\alpha e(F), \log f)$-biseparable graph with $d(F)\geq 1$ and $\log f > e^{K\kappa^2 r}$. Observe that $d(H)=rd(F)\ge K$.
\end{enumerate}
Then every $n$-vertex bipartite $(\varepsilon_1,\varepsilon_2 d)$-expander $G$ with $\delta(G) \geq d\geq \varepsilon e(H)$ contains a $TH$.

\end{lemma}

For these lemmas, we need to build a desired subdivision by finding a sequence of $x,y$-paths in the host graph $G$, where $xy\in E(H)$. For the dense case, we follow a standard application of the regularity lemma, and obtain an $\varepsilon$-regular partition. Then the biseparability of $H$ allows us to embed most of its edges in a regular pair from the regularity partition; for the remaining edges of $H$, we find disjoint short paths to replace them. For the sparse case, we shall use sublinear expanders again to embed an $H$-subdivision in $G$, following the proof strategy for Theorem \ref{main1}.

\medskip

Now we derive Theorem \ref{thm: sublog} and Lemma \ref{ans2} from Lemmas \ref{sparse}, \ref{dense2-3} and \ref{inte}. The proofs are essentially the same and for brevity we only present the latter (Lemma \ref{ans2}).
\begin{proof}[Proof of Lemma \ref{ans2}]
Take $\varepsilon_2=\frac{1}{10}, s=1600$. Let $\varepsilon_1$ be the constant obtained from Corollary \ref{expander}. Choose additional constants $c,K$ such  that  $\frac{1}{f}, \frac{1}{r} \ll \frac{1}{K} , \alpha\ll\frac{1}{K},c \ll  \frac{1}{\kappa}, \varepsilon, \varepsilon_1, \varepsilon_2$. Let $F$ be an $f$-vertex $\kappa$-degenerate $(\alpha e(F),\log f)$-biseparable graph and $H=F^{\Box r}$. Let $G$ be a graph with average degree $d(G)=d\geq\varepsilon e(H) =\frac{1}{2} \varepsilon r f^rd(F)$ and set $d_1:=\frac{d}{8}$. By Corollary \ref{expander}, $G$ has a bipartite $(\varepsilon_1,\varepsilon_2d_1)$-expander $G_1$ with $\delta(G_1)\geq d_1$. Now we let $|G_1|=n$. Using Lemma \ref{inte}, we obtain that if $\log^sn\leq d_1\leq \frac{n}{K}$, then $G\supseteq TH$. If $d_1\geq \frac{n}{K}$, then by Lemma \ref{dense2-3} with $\beta\geq K^{-1}$, we also have $G\supseteq TH$. Otherwise, by Lemma \ref{sparse}, $G\supseteq TK_{cd_1}^{(\ell)}$ for some  $\ell\in \mathbb{N}$. Since $r$ is sufficiently large and
$cd_1\geq \frac{1}{16} c\varepsilon r f^rd(F)\geq f^r$, we have a $TH$ in $G$.
\end{proof}

\section{Proof of Lemma \ref{dense}}\label{Thm3simi}
This section is devoted to a detailed proof of Lemma \ref{dense}, and our main task for building a balanced $H$-subdivision is to find a sequence of $x,y$-paths in the host graph $G$ whose lengths are exactly $\ell+1$. First in Section~\ref{sec:local-dense}, we reduce the problem to graphs that are locally dense. Then in Section \ref{web-unit}, we shall construct webs or units for building balanced subdivisions of $H$ in $G$. In Section \ref{secadj}, we introduce the concept of an \emph{adjuster}, which is a useful tool to adjust long paths with the required length. 

\subsection{Reduction to locally dense graphs}\label{sec:local-dense}

The next lemma is based on a simple yet powerful method known as Dependent Random Choice, which is comprehensively discussed in \cite{depran}. The \emph{codegree} of a pair of vertices $u,v$ in a graph, denoted as $d(u,v)$, is the number of their common neighbors.
\begin{lemma}\label{depranch}
Let $\alpha\in (0,1)$ and $G=(V_1,V_2)$ be a bipartite graph with $|V_i|=n_i$ for each $i\in[2]$ and $e(G)=\alpha n_1n_2$. If for $p,q\in \mathbb{N}$ it holds that $\alpha n_1>4(p+q)$ and $\alpha^2n_2>256q$, then $G$ contains a $TH^{(3)}$ for every $p$-vertex $q$-edge graph $H$.
\end{lemma}

\begin{proof}
Let $w\in V_2$ be a vertex chosen uniformly at random. Let $A$ denote the set of neighbors of $w$ in $V_1$, and define random variables $X=|A|$ and $Y$ as the number of pairs in $A$ with fewer than $4q$ common neighbors. Then
\begin{equation*}
\mathbb{E}[X]=\sum_{v\in V_1}\mathbb{P}[v\in X]=\sum_{v\in V_1}\frac{d(v)}{n_2}=\alpha n_1 \quad \text{and }\quad \mathbb{E}[Y]= \sum_{v_1,v_2} \mathbb{P}[w\in N(v_1)\cap N(v_2)]\leq \binom{n_1}{2}\cdot\frac{4q}{n_2},
\end{equation*}
where the later summation is over all pairs $\{v_1, v_2\} \in \binom{V_1}{2}$ with fewer than $4q$ common neighbors.
Using linearity of expectation, we obtain
\begin{equation*}
\mathbb{E}\left[X^2-\frac{\mathbb{E}[X]^2}{2\mathbb{E}[Y]}Y-\frac{\mathbb{E}[X]^2}{2}\right]\geq0.
\end{equation*}
Hence, there is a choice of $w$ such that this expression within the bracket is non-negative. Then
\begin{equation*}
X^2\geq\frac{1}{2}\mathbb{E}[X]^2> \frac{\alpha^2n_1^2}{2} \quad \text{and }\quad Y\leq 2\frac{X^2}{\mathbb{E}[X]^2}\mathbb{E}[Y]<\frac{4qX^2}{\alpha^2n_2}.
\end{equation*}
Consequently, $|A|=X>\frac{\alpha n_1}{2}$. Denote by $B$ the set of vertices $u\in A$ for which there are more than $\frac{|A|}{16}$ vertices $v\in A\setminus \{u\}$ satisfying $d(u,v)<4q$. Note that $|B|\leq\frac{32Y}{X}\leq \frac{128qX}{\alpha^2n_2}<\frac{X}{2}$ as $\alpha^2n_2>256q$.
\begin{figure}[H]
 \begin{center}
   \includegraphics[scale=0.45]{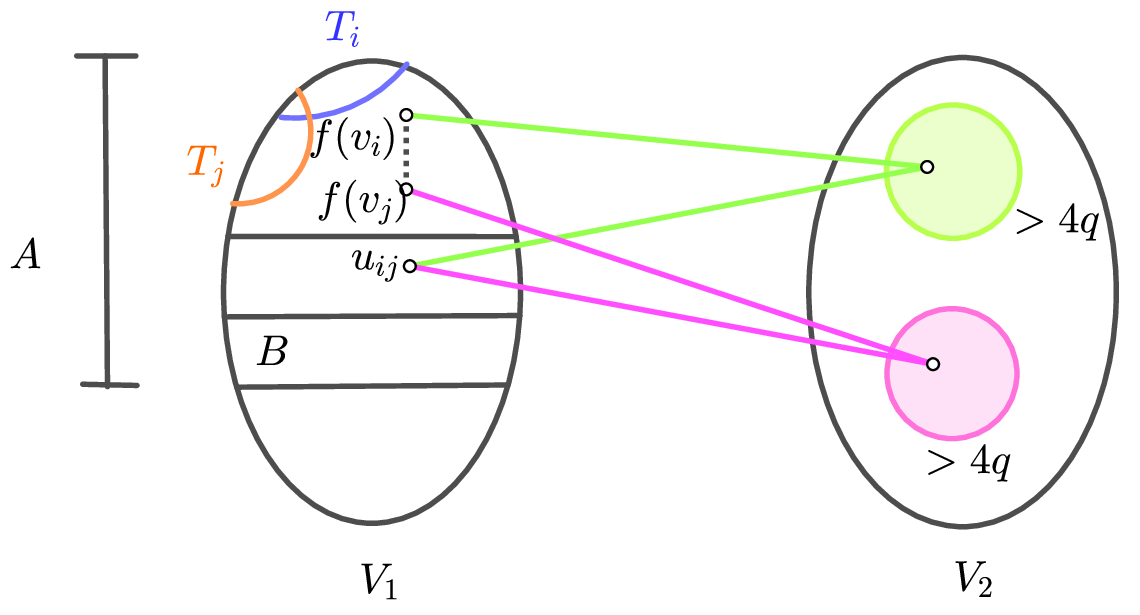}\\
 \caption{Embedding}
\label{dep}
 \end{center}
\end{figure}
\vspace{-0.4cm}
Now we shall embed all vertices of $H$ into $A\backslash B$ and replace each edge in $H$ by a copy of $P_5$ in $G$. Label the vertices of $H$ as $\{v_1,\ldots,v_{p}\}$. Let $f:V(H)\rightarrow A\backslash B$ be any injective mapping. Suppose $v_iv_j$ is the current edge for which we shall find a $f(v_i),f(v_j)$-path of length $4$ whilst avoiding all internal vertices used in previous connections. Let
$$T_i=\{u\in A\backslash B \mid d(u,f(v_i))<4q\} \quad \text{and }\quad T_j=\{u\in A\backslash B \mid d(u,f(v_j))<4q\}.$$ Then
$|T_i|,|T_j|\leq \frac{|A|}{16}$, and we pick a vertex in $A\backslash (B\cup T_i\cup T_j\cup\{f(v_i),f(v_j)\})$, say $u_{ij}$. Since there are at most $2(q-1)$ vertices in $V_2$ used in previous connections, by the choice of $u_{ij}$, we have $d(f(v_i),u_{ij}),~d(f(v_j),u_{ij})>4q>2(q-1)$, and thus one can pick two distinct vertices $x_i,x_j$ not used in previous connections to get the desired $f(v_i),f(v_j)$-path (see Figure \ref{dep}). As $|A\backslash B|>\frac{|A|}{2}$ and $|A\backslash B|-|T_i|-|T_j|-p\geq q$, there are enough vertices in $A\backslash(B\cup V(H))$ to serve as $u_{ij}$.
\end{proof}
We need the following crucial notion of denseness.
\begin{definition}\label{defn:locdense}
A graph $G$ is called \emph{locally-$(\alpha,\beta)$-dense} if
for every $W\subseteq V(G)$ with $|W|\leq \alpha$, we have $d(G-W)\geq \beta$.
\end{definition}

\begin{lemma}\label{aver}
Suppose $0< \frac{1}{K} \ll \frac{1}{x} < 1$ and $n,d$ and $q$ satisfy $n\geq Kd$ and $d\geq Kq$. Let $H$ be a $q$-edge graph and $G$ be an $n$-vertex graph with $\delta(G)\geq d$. If $G$ does not contain $TH^{(3)}$, then $G$ is locally-$(dm^x,\frac{d}{2})$-dense, where $m= \log^4 \frac{n}{d}$.
\end{lemma}
\begin{proof}
Fix $W\subseteq V(G)$ with $|W|\leq dm^x$. As $\delta(G)\geq d$, $d(G-W)\geq \delta(G)-|W|\geq \frac{d}{2}$ when $|W|\leq \frac{d}{2}$. We may assume $|W|> \frac{d}{2}$. Suppose to the contrary that $d(G-W)< \frac{d}{2}$, then
\begin{equation*}
e(V(G-W),W)=\sum_{v\in V(G-W)}d(v)-2e(G-W)\geq \frac{d}{2}(n-|W|).
\end{equation*}
Let $\alpha=\frac{e(V(G-W),W)}{(n-|W|)|W|}$. Then $\alpha|W|=\frac{e(V(G-W),W)}{n-|W|}\geq \frac{d}{2}>4(|V(H)|+q)$ as $|V(H)|\leq 2q$ and
\begin{equation}\label{n2}
\alpha^2(n-|W|)\geq \frac{d^2(n-|W|)}{4|W|^2}>\frac{n-|W|}{4m^{2x}}.
\end{equation}
Since $n\geq Kd$, we get $dm^{2x}\leq \frac{n}{2}$ and $n-|W|>\frac{n}{2}$. Thus, (\ref{n2}) implies that $\alpha^2(n-|W|)>\frac{d}{4}>256q$. Hence, applying Lemma \ref{depranch} with $W, V(G)-W$ playing the roles of $V_1,V_2$, respectively, we can find a copy of $TH^{(3)}$ in $G$, a contradiction.
\end{proof}

\subsection{Constructing units and webs}\label{web-unit}
For $x\in \mathbb{N}$, an \emph{$x$-star} is a star with $x$ leaves.
\begin{definition}[unit]\label{defn:unit}
For $h_1,h_2,h_3\in \mathbb{N}$, a graph $F$ is an \emph{$(h_1,h_2,h_3)$-unit} if it contains distinct vertices $u$ (the \emph{core} vertex of $F$) and $x_1,\ldots,x_{h_1}$, and $F=\bigcup_{i\in[h_1]}(P_i\cup S_{i})$, where
\begin{itemize}
\item $\mathcal{P}=\bigcup_{i\in[h_1]}P_i$ is a collection of pairwise internally vertex-disjoint paths, each of length at most $h_3$, such that $P_i$ is a $u,x_i$-path, and
\item $\mathcal{S}=\bigcup_{i\in[h_1]}S_{i}$ is a collection of vertex-disjoint $h_2$-stars such that $S_{i}$ has center $x_i$ and $\bigcup_{i\in[h_1]}(V(S_{i})\backslash \{x_i\})$ is disjoint from $V(\mathcal{P})$.
\end{itemize}
\end{definition}
We call $S_{i}$ a \emph{pendent} star in the unit $F$ and every such path $P_i$ a \emph{branch} of $F$. Define the \emph{exterior} $\mathsf{Ext}(F):=\bigcup_{i\in[h_1]}(V(S_{i})\backslash \{x_i\})$ and \emph{interior} $\mathsf{Int}(F):=V(F)\backslash \mathsf{Ext}(F)$.
\begin{figure}[H]
 \begin{center}
   \includegraphics[scale=0.29]{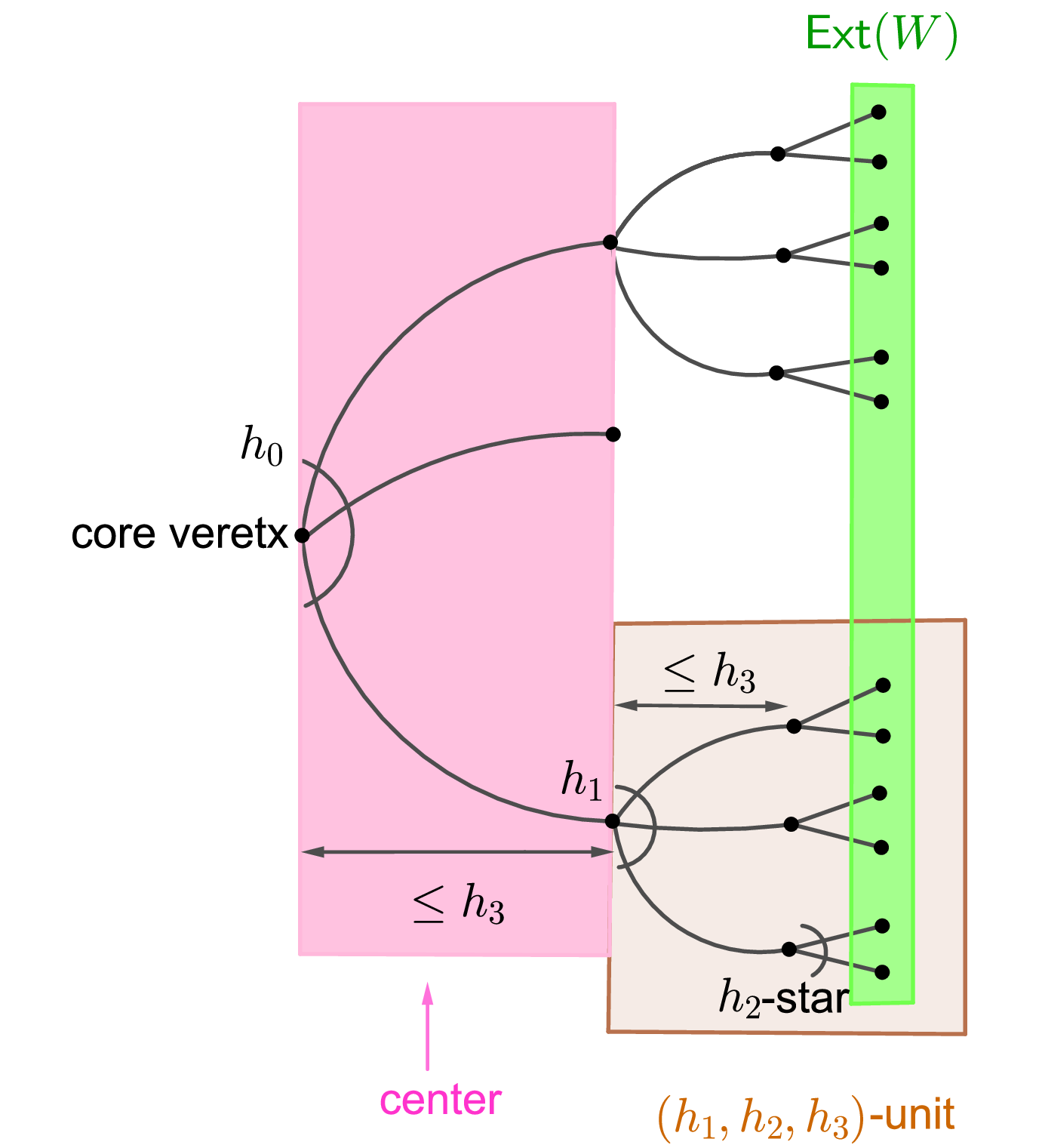}\\
   \caption{$(h_0,h_1,h_2,h_3)$-web}
\label{webdraw}
 \end{center}
\end{figure}
\begin{definition}[web]\label{defn:web}
For $h_0,h_1,h_2,h_3\in \mathbb{N}$, a graph $W$ is an \emph{$(h_0,h_1,h_2,h_3)$-web} if it contains distinct vertices $v$ (the \emph{core} vertex of $W$), $u_1,\ldots,u_{h_0}$, and $W=\bigcup_{i\in[h_0]}(Q_i\cup F_{i})$, where
\begin{itemize}
\item $\mathcal{Q}=\bigcup_{i\in[h_0]}Q_i$ is a collection of pairwise internally vertex-disjoint paths such that each $Q_i$ is a $v,u_i$-path of length at most $h_3$, and
\item $\mathcal{F}=\bigcup_{i\in[h_0]}F_{i}$ is a collection of vertex-disjoint $(h_1,h_2,h_3)$-units such that $F_{i}$ has core vertex $u_i$ and $\bigcup_{i\in[h_0]}(V(F_{i})\backslash \{u_i\})$ is vertex-disjoint from $V(\mathcal{Q})$.
\end{itemize}
\end{definition}

We call each $Q_i$ a \emph{branch} and call the branches inside each unit $F_i$ the \emph{second-level} branches of $W$. Similarly, we define the \emph{exterior} $\mathsf{Ext}(W):=\bigcup_{i\in[h_0]}\mathsf{Ext}(F_{u_i})$, and the \emph{interior} $\mathsf{Int}(W):=V(W)\backslash \mathsf{Ext}(W)$ and additionally define the \emph{center} $\mathsf{Ctr}(W):=V(\mathcal{Q})$.

We need two technical results that enable us to find a collection of units (for the vertices in $\mathbf{L}$) and webs (for the vertices in $\mathbf{M}\cup \mathbf{S}$) as anchoring points for building a balanced subdivision of $H$.

\begin{lemma}\label{middlewebs}
Suppose $\frac{1}{n}, \frac{1}{d}\ll \frac{1}{K} \ll \varepsilon_1,\varepsilon_2 < \frac{1}{5}$ and $x,y,z\in \mathbb{N}$ satisfy $\frac{y-9}{2}< z<y<\min\{x,z+10\}$. Let $n,d,\gamma$ be integers satisfying $m^x \leq d\leq \frac{n}{K}$ and $m^z\leq \gamma < \frac{d}{m^{10}}$ where $m= \log^4 \frac{n}{d}$. If $G=(V_1,V_2,E)$ is a locally-$(dm^x, \frac{d}{2})$-dense bipartite $(\varepsilon_1,\varepsilon_2d)$-expander with $d(G)=d$ and $W$ is a set of vertices with $|W|\leq 100 dm^{x-2y+z-4}$, then $G-W$ contains a $(22\gamma, m^{y-z}, \frac{dm^z}{20\gamma}, 4m)$-web with core vertex lying in $V_1$.
\end{lemma}

Lemma~\ref{middlewebs} can be proved by using the strategy in~\cite{KCLiu}, and we provide a detailed proof in the Appendix \ref{appeBwebs}. 
However, our construction of webs in Lemma~\ref{middlewebs} is very delicate for the relatively large value of $\gamma$, and it is still unclear how to build a web for $\gamma=\Omega(d)$ as above. To handle this, we use the following result to build large units instead.

\begin{lemma}\label{bunit}
Suppose $x,y,z,s\in \mathbb{N}$ satisfy $s\geq \max\{8x,y\}$ and $\frac{1}{n}, \frac{1}{d}\ll \frac{1}{K} \ll c_0 \ll  \varepsilon_1, \varepsilon_2, \tfrac{1}{x}, \tfrac{1}{y}, \tfrac{1}{z}, \tfrac{1}{s}$ such that $\log^s n \leq d \leq \frac{n}{K}$. Let $m= \log^4 \frac{n}{d}$. Let $G=(V_1,V_2,E)$ be an $n$-vertex bipartite $(\varepsilon_1, \varepsilon_2d)$-expander with $dm^x\geq \Delta(G) \geq \delta(G)\geq d$ and $W$ be a vertex set with $|W|\leq dm^z$. Then $G-W$ contains a $(c_0d, m^y, 2m)$-unit with core vertex lying in $V_1$.
\end{lemma}

Previous approaches to building units usually proceed by first linking many disjoint stars and then using averaging arguments, which can only produce units with sublinear in $d(G)$ many branches. Here we directly employ the robust expansion property to construct the desired large unit with a linear number of branches in Lemma~\ref{bunit}. For this, we need the following notion.
\begin{definition}[\cite{LiuC4}]
Given a graph $G$ and $W\subseteq V(G)$, we say that paths $P_1,\ldots,P_t$, each starting with a vertex $v$ and contained in the vertex set $W$, are \emph{consecutive shortest paths from $v$ in $W$} if for each $i$ $(1\leq i\leq t)$, the path $P_i$ is a shortest path between its endpoints in the set $W\setminus(\bigcup_{j<i}V(P_j))\cup\{v\}$.
\end{definition}
Note that a collection of consecutive shortest paths from a vertex $v$ may not be unique. In particular, there can be multiple choices of a shortest path $P_i$ in $W\setminus(\bigcup_{j<i}V(P_j))\cup\{v\}$. 

The robust expansion property we need is formulated as follows, which enables us to find many consecutive shortest paths from a fixed vertex. We defer its proof to Appendix~\ref{app1}.
\begin{lemma}\label{ball}
Suppose $s,x\in \mathbb{N}$ satisfy $s\geq 8x$ and $\frac{1}{n}, \frac{1}{d}\ll c, \frac{1}{K} \ll \varepsilon_1, \varepsilon_2, \tfrac{1}{s}<\frac{1}{5}$ such that 
$\log^s n < d<\frac{n}{K}$. Let $m= \log^4 \frac{n}{d}$. Let $H$ be an $n$-vertex and $(\varepsilon_1,\varepsilon_2d)$-expander with $\delta(H)\geq \frac{d}{2}$ and $P_1,\dots, P_t$ be consecutive shortest paths from $v$ in $B^{m}_H(v)$. Writing $U=\bigcup_{i\in t}V(P_i)$, if $t\leq cd$, then $|B_{H-(U\backslash\{v\})}^m(v)|\geq dm^x$.
\end{lemma}

\begin{proof}[Proof of Lemma \ref{bunit}]
Given $\varepsilon_1,\varepsilon_2,x,y,z,s$ such that $s\geq \max\{8x,y\}$, we are given constants such that $\tfrac{1}{n},\tfrac{1}{d}\ll\frac{1}{K}\ll c_0\ll \varepsilon_1, \varepsilon_2, \frac{1}{x}, \frac{1}{y}, \frac{1}{z}, \tfrac{1}{s}$ and take $t=|y-z|+4$. Let $G=(V_1,V_2,E)$ be a bipartite $(\varepsilon_1,\varepsilon_2d)$-expander with $\delta(G)\geq d$, $\Delta(G)\leq dm^x$ and $W\subseteq V(G)$ with $|W|\leq dm^z$. We first greedily find $dm^{t}$ vertex-disjoint stars $S_1,\ldots,S_{dm^{t}}$, each with $2m^{y}$ leaves. This can indeed be done by picking an average vertex as $\Delta(G)\leq dm^x$ and $\Delta(G)\cdot dm^t\cdot 2m^y<n\cdot \delta(G)/10\le e(G)/5$.
Denote by $u_i$ the center vertex of $S_i$ for each $i\in[dm^{t}]$ and $U:=\{u_1,\ldots,u_{dm^{t}}\}$. Note that $|W|+|\bigcup_{i\in[dm^{t}]}V(S_i)|\leq 2dm^{y+t}+dm^z$. Applying Lemma \ref{ave2} with $G, \bigcup_{i\in[dm^{t}]}V(S_i)\cup W$ playing the roles of $H,X$, we have that $G-W-\bigcup_{i\in[dm^{t}]}V(S_i)$ contains a set $Y_1$ such that 
$$|Y_1|\geq n-\tfrac{4dm^{x+y+t}+2dm^{x+z}}{\rho(n)d(G)}\geq n-\tfrac{6dm^{x+y+t}}{\rho(n)d(G)}\geq n-\tfrac{6}{\varepsilon_1}\log^{4(x+y+t)}(\tfrac{n}{d})\log^2(\tfrac{15n}{\varepsilon_2d})\geq\tfrac{n}{2},$$ 
and $G[Y_1]$ is an $(\varepsilon_1,\varepsilon_2d)$-expander with $d(G[Y_1])\geq \frac{d}{2}$. Next we will find a desired ball in $G_1:=G[Y_1]$. Arbitrarily choose a vertex of degree $\frac{d}{2}$ in $G_1$. Then by Proposition \ref{proball}, there exists a ball $B_{G_1}^m(v)$ in $G_1$ such that $$|B_{G_1}^m(v)|\geq \frac{|V(G_1)|}{2}\geq\frac{n}{4} \geq dm^{y+t}.$$
To build the desired $(c_0d,m^{y},2m)$-unit, we shall proceed by finding $2c_0d$ internally vertex-disjoint paths $Q_1,\ldots,Q_{2c_0d}$ in $G$ from $v$ satisfying the following rules.
\stepcounter{propcounter}
\begin{enumerate}[label = ({\bfseries \Alph{propcounter}\arabic{enumi}})]
\rm\item\label{ballab1} Each path is a $v,u_i$-path of length at most $2m$.
\rm\item\label{ballab2} None of the paths contains any vertex in $U\cup W$ as an internal vertex.
\rm\item\label{ballab3} The subpaths $Q_i[B^m_{G_1}(v)]$, $i\in[2c_0d]$, form consecutive shortest paths from $v$ in $B^m_{G_1}(v)$.
\end{enumerate}

Assume that we have iteratively obtained a collection of shortest paths $\mathcal{Q}=\{Q_1,\ldots,Q_{s'}\}$ $(0\leq s'<2c_0d)$ as in \ref{ballab1}-\ref{ballab3}. Then $|\mathsf{Int}(\mathcal{Q})|<4c_0dm$. Note that \ref{ballab3} gives $s'$ consecutive shortest paths $P_1,\ldots,P_{s'}$ from $v$ in $B_{G_1}^m(v)$, where $P_i=Q_i[B^m_{G_1}(v)]$. Write $\mathcal{P}=\{P_1,\ldots,P_{s'}\}$. Applying Lemma \ref{ball} to $G_1$, we get
$$|B^m_{G_1-\mathsf{Int}(\mathcal{Q})}(v)|=|B_{G_1-\mathsf{Int}(\mathcal{P})}^m(v)|\geq dm^{y+t}.$$
Let $U'$ be the set of leaves of all stars $S_i$ whose centers are not used as endpoints of paths $Q_i$ for all $i\in[s']$ and $U'\cap \mathsf{Int}(\mathcal{Q})=\varnothing$. Then we have
$$|U'|\geq 2m^{y}(dm^{t}-2c_0d)-|\mathsf{Int}(\mathcal{Q})|>dm^{y+t}.$$
Note that
$$|W|+|V(\mathcal{Q})|+|U|\leq dm^z+4c_0dm+dm^{t}<2c_0dm^{y+t-1}. $$
Applying Lemma \ref{distance} with $B_{G_1-\mathsf{Int}(\mathcal{Q})}^m(v)$, $U'$, $V(\mathcal{Q})\cup U\cup W$ playing the roles of $X_1,X_2,W$, respectively, we can find a shortest path, say $Q'$ from $B_{G_1-\mathsf{Int}(\mathcal{Q})}^m(v)$ to some $u_j$, and write $w'$ for the endpoint of $Q'$ inside the ball $B^m_{G_1-\mathsf{Int}(\mathcal{Q})}(v)$. Then $B^m_{G_1-\mathsf{Int}(\mathcal{Q})}(v)\cap V(Q')=\{w'\}$ and one can easily find a $v,w'$-path, denoted as $P_{s'+1}$, inside $B^m_{G_1-\mathsf{Int}(\mathcal{Q})}(v)$. Let $Q_{s'+1}=P_{s'+1}Q'$ be the concatenation of the two paths $P_{s'+1}$ and $Q'$. Then the paths $Q_1,\ldots,Q_{s'+1}$ satisfy \ref{ballab1}-\ref{ballab3}. Repeating this for $k=0,1,\ldots,2c_0d$, yields $2c_0d$ paths $Q_1,\ldots, Q_{2c_0d}$ as desired.

Let $W'=\bigcup_{i\in[2c_0d]}V(Q_i)$. Then $|W'|\leq 4c_0dm$.
For every $i\in [dm^{t}]$, we say $S_i$ is \emph{overused} if at least $m^{y}$ leaves of $S_i$ are used in $W'$. Then there are at most $\frac{4c_0dm}{m^{y}}$ overused stars. Hence, we have at least $2c_0d-\frac{4c_0d}{m^{y-1}}>c_0d$ stars that are not overused, say $S_1,\ldots S_{c_0d}$, such that their centers are connected to $v$ via the paths $Q_i$ as above. Then these stars together with the corresponding paths $Q_i$ yield a $(c_0d,m^{y},2m)$-unit as desired.
\end{proof}

\subsection{Constructing adjusters}\label{secadj}
Given a graph $F$ and a vertex $v\in V(F)$, we say $F$ is a \emph{$(D,m)$-expansion centered at $v$} if $|F|=D$ and $v$ is at distance at most $m$ in $F$ from any other vertex of $F$.

\begin{definition}[\cite{Adj}]\label{defn:adjuster}
A \emph{$(D,m,k)$-adjuster} $\mathcal{A}=(v_1,F_1,v_2,F_2,A)$ in a graph $G$ consists of vertices $v_1,v_2\in V(G)$, graphs $F_1,F_2\subseteq G$ such that the following holds for some $\ell\in \mathbb{N}$.
\stepcounter{propcounter}
\begin{enumerate}[label = ({\bfseries \Alph{propcounter}\arabic{enumi}})]
\rm\item\label{adjla1} $A$, $V(F_1)$ and $V(F_2)$ are pairwise disjoint.
\rm\item\label{adjla2} For each $i\in[2]$, $F_i$ is a $(D,m)$-expansion centered at $v_i$.
\rm\item\label{adjla3} $|A|\leq 10mk$.
\rm\item\label{adjla4} For each $i\in\{0,1,\ldots,k\}$, there is a $v_1,v_2$-path in $G[A\cup\{v_1,v_2\}]$ with length $\ell+2i$.
\end{enumerate}
\end{definition}

We denote by $\ell(\mathcal{A})$ the smallest integer $\ell$ for which \ref{adjla4} holds. Note that $\ell(\mathcal{A})\leq |A|+1\leq 10mk+1$. We refer to the graphs $F_1$ and $F_2$ of an adjuster $\mathcal{A}=(v_1,F_1,v_2,F_2,A)$ as the \emph{ends} of the adjuster, and let $V(\mathcal{A})=V(F_1)\cup V(F_2)\cup A$. Moreover, $v_1,v_2$ are called \emph{core} vertices of $\mathcal{A}$, and $A$ is called the \emph{center} vertex set of $\mathcal{A}$. We call a $(D,m,1)$-adjuster a \emph{simple} adjuster.

We will need the following variations of lemmas from~\cite{Adj}. As the proofs for these variations are almost identical to that of the original, we provide detailed proofs in the Appendix \ref{appen-adj}.

The first lemma robustly finds inside a locally dense expander an adjuster.

\begin{lemma}\label{link}
Suppose $\frac{1}{n},\frac{1}{d}\ll \frac{1}{K} \ll \varepsilon_1,\varepsilon_2<\frac{1}{5}$ and $s,x,y\in \mathbb{N}$ satisfy $s\geq 1600$, $s\geq 8x>8y$ and $\log^s n < d< \frac{n}{K}$. Let $m=\log^4\frac{n}{d}$ and $D= 10^{-7}dm^{y}$. If $G$ is an $n$-vertex locally-$(dm^x, \frac{d}{2})$-dense $(\varepsilon_1, \varepsilon_2d)$-expander with $\delta(G)\geq d$ and $W\subseteq V(G)$ with $|W|\leq m^{-\frac{3}{4}}D$, then $G-W$ contains a $(D,m,r)$-adjuster for any $r\leq 10^{-1}dm^{y-2}$.
\end{lemma}

The second lemma below finds a path of length in a controlled short window, which in conjunction with Lemma~\ref{link} provides a path of a fixed desired length between two not too small sets ($Z_1,Z_2$ below). 

\begin{lemma}\label{extadju}
Suppose $\frac{1}{n}, \frac{1}{d}\ll \frac{1}{K}\ll \varepsilon_1, \varepsilon_2<1$  and $s,x,y\in \mathbb{N}$ satisfy $s\geq 1600$, $s\geq 8x>8y$ and $\log^s n < d < \frac{n}{K}$. Let $m = \log^4 \frac{n}{d}$ and $D=10^{-7}dm^{y}$ and $\ell \leq d m^{y-2}$. Suppose that $G$ is an $n$-vertex locally-$(dm^x, \frac{d}{2})$-dense $(\varepsilon_1, \varepsilon_2d)$-expander with $\delta(G)\geq d$ and the following conditions hold.
\begin{enumerate}
\item[$(1)$] $W\subseteq V(G)$ with $|W|\leq m^{-\frac{3}{4}}D$.
\item[$(2)$] $Z_i\subseteq V(G)\setminus W$ are pairwise disjoint vertex sets of size at least $D$ for each $i\in [2]$.
\item[$(3)$] $I_j\subseteq V(G)\setminus(W\cup Z_1\cup Z_2)$ are vertex-disjoint $(D,m)$-expansion centered at some vertex $v_j$ for each $j\in [2]$.
\end{enumerate}
Then $G-W$ contains vertex-disjoint paths $P$ and $Q$ with $\ell \leq \ell(P)+\ell(Q)\leq \ell+18m$ such that $P,Q$ link $v_1,v_2$ to some vertices $z_1\in Z_1, z_2\in Z_2$, respectively.
\end{lemma}

\subsection{Warm-up: many large degree vertices}
Let $L_G:=\{v\in V(G):d_G(v)\geq 2dm^{12}\}$, where $m=\log^4\frac{n}{d}$. In this subsection, we consider the case when $|L_G|\geq 4q$.
\begin{lemma}\label{many}
Suppose $\frac{1}{n},\frac{1}{d}\ll \frac{1}{K} \ll \varepsilon_1,\varepsilon_2, \frac{1}{s}<1$ and $s\geq 400$, $q\in \mathbb{N}$ satisfy $\log^s n \leq d\leq \frac{n}{K}$ and $q\leq \frac{d}{K}$. Let $m=\log^4\frac{n}{d}$. Let $G$ be an $n$-vertex bipartite $(\varepsilon_1,\varepsilon_2d)$-expander with $\delta(G)\geq d$ and $H$ be a $q$-edge graph. If $|L_G|\geq 4q$, then $TH^{(\ell)}\subseteq G$ for some $\ell\in \{3,m^3\}$.
\end{lemma}
\begin{proof}
We may assume that $G$ is $TH^{(3)}$-free, otherwise we are done. 
Applying Lemma \ref{aver} to $G$ with $x=50$, we have that $G$ is locally-$(dm^{50},\frac{d}{2})$-dense.
Let $V(H)=\{x_1,\ldots,x_h\}$ and $E(H)=\{e_1,\ldots,e_q\}$ with $q\leq \frac{d}{K}$. Since $H$ has no isolated vertices, we obtain $q\geq \frac{h}{2}$, and so $|L_G|\geq \frac{h}{2}\cdot4=2h$. Thus, it is possible to take a set $Z=\{u_1,\ldots,u_h\}$ of $h$ distinct vertices in $L_G$ such that all vertices in $Z$ lie in the same part of $G$. Let $\tau: V(H)\rightarrow Z$ be an arbitrary injection. Note that for each $i\in[h]$, the set $N(u_i)$ has size at least $2dm^{12}$. Next we shall construct a $TH^{(\ell)}$ by greedily finding a collection of paths of the same length $\ell =m^3$. Assume that we have a maximal collection of pairwise internally disjoint paths, say $P(e_1),\ldots, P(e_t)$, such that $t\le q$ and each $P(e_j)$ is a path of length exactly $\ell$ in $G$ connecting the two vertices in $\tau(e_j)$ whilst $P(e_j)$ is internally disjoint from $Z$. We claim that $t=q$ and so these paths $P(e_j)$ yield a balanced subdivision of $H$. Suppose for contradiction that $t<q$. We shall find one more path $P(e_{t+1})$ for $e_{t+1}\in E(H)$. Write $e_{t+1}=x_1x_2$ and let $u_i=\tau(x_i)$ for $i\in[2]$.

Let $W=\bigcup_{j\in[t]}\mathsf{Int}(P(e_j))$ be the union of the internal vertices of the paths. Then $|W|+|Z|< q\cdot \ell+h\leq 2dm^3$. Set $D=10^{-7}dm^{12}$ and we have $|W\cup Z|<\frac{1}{2}m^{-\frac{3}{4}}D$. Applying Lemma \ref{link} with $y=12$, $r=20m$ and $W\cup Z$ playing the role of $W$, we obtain a $(D,m,20m)$-adjuster, say $\mathcal{A}=(v_{1},F_1,v_2,F_2,A)$ in $G-(W\cup Z)$, and observe that $\ell(\mathcal{A})\leq |A|\leq 200m^2$. Note that
\[|N(u_{i})\setminus(V(\mathcal{A})\cup W\cup Z)|\geq 2dm^{12}-2D-200m^2-qm^3-h\geq 2D~\text{for each}~ i\in[2].\]
Then there are disjoint vertex sets $U_{1}\subseteq N(u_{1})$ and $U_{2}\subseteq N(u_{2})$ each of size $D$ in $G-(V(\mathcal{A})\cup W\cup Z)$. Choose $\ell'=\ell-19m-\ell(\mathcal{A})$. Since $d\geq \log^sn$ and $|W\cup A\cup Z|\leq 200m^2+\frac{1}{2}m^{-\frac{3}{4}}D+q\leq m^{-\frac{3}{4}}D$, by applying Lemma \ref{extadju} with $x=50$ and $W_{\ref{extadju}}=W\cup A\cup Z$, there exist vertex-disjoint paths $P$ and $Q$ linking $\{y_1,y_2\}$ to $\{v_1,v_2\}$ for $y_1\in U_{1}, y_2\in U_{2}$, and $\ell'\leq \ell(P)+\ell(Q)\leq \ell'+18m$. We may assume that $P$ is a $y_1,v_1$-path and $Q$ is a $y_2,v_2$-path.
Then $P'=\{u_{1}y_1\}\cup P$ is a $u_{1},v_1$-path and $Q'=\{u_{2}y_2\}\cup Q$ is a $u_{2},v_2$-path with $\ell'\leq \ell(P')+\ell(Q')\leq \ell'+19m$. Also observe that $\ell(\mathcal{A})\leq \ell-\ell(P')-\ell(Q')\leq \ell(\mathcal{A})+19m$. Since $u_1,u_2$ lie in the same part of $G$, we obtain that $\ell(\mathcal{A})$ and $\ell(P')+\ell(Q')$ have the same parity (modulo 2). Furthermore, since $\ell$ is even and $\mathcal{A}$ is a $(D,m,20m)$-adjuster, it follows by definition that $A$ contains a $v_1,v_2$-path say $R'$ of length $\ell-\ell(P')-\ell(Q')$.
Thus, the path $P'\cup R'\cup Q'$, denoted as $P(e_{t+1})$, has length $\ell$ and connects $u_{1}$ and $u_{2}$ whilst avoiding $W\cup Z$, which together with $\{P(e_{1}),\ldots,P(e_{t})\}$ contradicts the maximality of $t$.
\end{proof}

\subsection{Putting things together: proof of Lemma \ref{dense}}\label{togeth}
We need the following result of Fox and Sudakov, which is another application of  Dependent Random Choice.
\begin{lemma}[\cite{depran}]\label{dep1}
Let $H$ be a graph with at most $n$ edges and vertices and let $G$ be a graph with $N$ vertices and $\varepsilon N^2$ edges such that $N>128\varepsilon^{-3}n$. Then $TH^{(1)}\subseteq G$.
\end{lemma}

\begin{proof}[Proof of Lemma \ref{dense}]
Given $0<\varepsilon_1,\varepsilon_2<\frac{1}{5}$, we choose additional constants $K,c_0$ such that
$$\frac{1}{n}, \frac{1}{d}\ll \frac{1}{C}\ll \frac{1}{K}\ll c_0\ll \varepsilon_1,\varepsilon_2,\frac{1}{s}.$$
Let $G$ be an $n$-vertex bipartite $(\varepsilon_1,\varepsilon_2d)$-expander with $\delta(G)\geq d\ge \log^s n$ and $H$ be a $q$-edge graph with $q\leq \frac{d}{C}$. If $d>\frac{n}{K}$, then as $\frac{1}{C}\ll\frac{1}{K}$ and $d\geq Cq$, we have $n\geq 128(2K)^3\cdot 2q>2048K^2(|V(H)|+q)$. Applying Lemma \ref{dep1} with $\varepsilon=\frac{1}{2K}$, we get $TH^{(1)}\subseteq G$. Hence it remains to consider the case when $\log^s n\leq d\leq \frac{n}{K}$. Recall that $L_G:=\{v\in V(G):d_G(v)\geq 2dm^{12}\}$, where $m=\log^4\frac{n}{d}$. If $|L_G|\geq 4q$, then Lemma \ref{many}  gives us a $TH^{(\ell)}$ for some $\ell\in\{3,m^3\}$. Now we assume that $|L_G|< 4q\leq \frac{4d}{C}< \frac{\varepsilon_2d}{2}$. We may further assume that $G$ is $TH^{(3)}$-free, otherwise we are done. Applying Lemma \ref{aver} to $G$ with $x=50$, we have that $G$ is locally-$(dm^{50},\frac{d}{2})$-dense.

In the rest of the proof, we take
$\ell=m^3$ and our goal is to embed a copy of $TH^{(\ell)}$.
We call $Z$ an \emph{object} if it is a web or unit. We partition $V(H)$ into three parts as follows:
$$
\mathbf{L}=\{v\in V(H)|\ d(v)\geq \tfrac{d}{m^{10}}\}, \quad \mathbf{M}=\{v\in V(H)|\ m^4<d(v)< \tfrac{d}{m^{10}}\}, \quad \mathbf{S}=\{v\in V(H)|\ d(v)\leq m^4\}.
$$
In the following, 
for every $v\in V(H)$, we shall construct a web or a unit in $G$ (depends on the degree of $v$), such that these objects are pairwise internally disjoint.
Note that
\begin{equation*}
2e(H)=\sum_{v\in V(H)}d_H(v)\geq |\mathbf{L}|\cdot \tfrac{d}{m^{10}},
\end{equation*}
 and thus $|\mathbf{L}|\leq m^{10}$.

First greedily find a family of internally vertex-disjoint webs $\{Z_v\}_{v\in \mathbf{M}}$, where each $Z_v$ is a $(22d_H(v),m^{8},\frac{dm^4}{20d_H(v)},4m)$-web and $2|\mathbf{S}|$ internally vertex-disjoint $(22m^4,m^{8},\frac{d}{20},4m)$-webs, say $Z_1,\ldots,Z_{2|\mathbf{S}|}$. This can be done by repeatedly applying Lemma \ref{middlewebs} to $G$ with $x=50, y=12, z=4$ and $W$ being the set of internal vertices of objects found so far, since
\begin{align*}
    |W| &\leq \sum_{v\in \mathbf{M}}22d_H(v)\cdot (4m+m^8\cdot 4m) + 2|\mathbf{S}|\cdot 22m^{13} \leq 
\sum_{v\in\mathbf{M}}90 d_H(v)m^9+44|\mathbf{S}|m^{13} \\&\le 180q m^{9}+44qm^{13}<100dm^{13}.
\end{align*}

\begin{claim}\label{subexpan}
$G_1:=G-L_G$ is an $(\frac{\varepsilon_1}{2},\varepsilon_2d)$-expander satisfying $\delta(G_1)\geq \frac{d}{2}$ and $|G_1|\geq \frac{n}{2}$.
\end{claim}
\begin{proof}\renewcommand*{\qedsymbol}{$\blacksquare$}
Recall that $|L_G|\leq \frac{4d}{C}<\delta(G)<|G|$, we know that $L_G\neq V(G)$. Therefore $|G_1|\geq n-|L_G|\geq n-\frac{4d}{C}\geq \frac{n}{2}$. Furthermore, $\delta(G_1)\geq \delta(G)-|L_G|\geq \frac{d}{2}$. To finish the proof of the claim, it is left to show that $G_1$ is an $(\frac{\varepsilon_1}{2},\varepsilon_2d)$-expander. Since $G$ is an $(\varepsilon_1,\varepsilon_2d)$-expander and $\rho(x)x$ is increasing when $x\geq \frac{\varepsilon_2d}{2}$, for any set $X$ in $G_1$ of size $x\geq \frac{\varepsilon_2d}{2}$ with $x\leq \frac{|G_1|}{2}\leq \frac{|G|}{2}$, we have
$$
|N_G(X)|\geq x\cdot\rho(x,\varepsilon_1,\varepsilon_2d)\geq \tfrac{\varepsilon_2d}{2}\cdot\rho\left(\tfrac{\varepsilon_2d}{2},\varepsilon_1,\varepsilon_2d\right)
=\tfrac{\varepsilon_2d}{2}\cdot \tfrac{\varepsilon_1}{\log^2(\tfrac{15}{2})} \geq \tfrac{\varepsilon_1\varepsilon_2d}{10}\geq \tfrac{8d}{C} \geq 2|L_G|.
$$
Hence, $|N_{G_1}(X)|\geq |N_G(X)|-|L_G|\geq \frac{1}{2}|N_G(X)|\geq \frac{1}{2}x\cdot\rho(x,\varepsilon_1,\varepsilon_2d)=x\cdot\rho\left(x,\frac{\varepsilon_1}{2},\varepsilon_2d\right)$ as desired.
\end{proof}
Applying Lemma \ref{bunit} on $G_1$ with $x=14, y=13, z=14$, we can greedily pick a family $\{Z_v\}_{v\in \mathbf{L}}$ of units such that $Z_v$ is a $(c_0d,m^{13},2m)$-unit and 
the interiors of all units and webs obtained as above are disjoint.
This is possible because in the process, the union of $L_G$ and the interiors of all possible units and webs have size at most $dm^{14}$.

Denote by $z_{v}$ the core vertex of the object $\{Z_v\}_{v\in \mathbf{M}\cup \mathbf{L}}$ and $z_i$ the core vertex of the web $Z_i$ for each $i\in[2|\mathbf{S}|]$. Recall that all these core vertices lie in the same part $V_1$. Moreover, any two objects can only overlap at their exteriors. Note that
$$|\mathsf{Ext}(Z_v)|=\begin{cases}
c_0dm^{13}, \quad \text{if } v\in\mathbf{L},\\
\frac{11dm^{12}}{10}, \quad \text{if } v\in\mathbf{M},
\end{cases} \text{and} \quad |\mathsf{Ext}(Z_i)|=\tfrac{11dm^{12}}{10}, \quad \text{if } i\in[2|\mathbf{S}|].$$

Let $H_1$ be the subgraph of $H$ with $V(H_1)=V(H)$ and \[E(H_1)=E(H[\mathbf{S}])\cup E(H[\mathbf{S},\mathbf{L}\cup\mathbf{M}]),\] that is, all edges that touch $\mathbf{S}$ and write $H_2=H\backslash E(H_1)$. We shall find a mapping $f:V(H)\rightarrow V(G)$ and a family of pairwise internally disjoint paths of the same length $\ell$ respecting the adjacencies of $H$ in the following two rounds, where we may abuse the notation $f$ to mean the up-to-date embedding. We begin with embedding every $v\in\mathbf{L}\cup \mathbf{M}$ by taking $f(v)=z_v$.

\medskip
\noindent \textbf{First round: Finding the desired paths (in $G$) for the adjacencies in $H_1$.}\medskip

\noindent Let $W=(\bigcup_{v\in \mathbf{L}\cup\mathbf{M}} \mathsf{Int}(Z_v))\cup(\bigcup_{i\in[2|\mathbf{S}|]}\mathsf{Ctr}(Z_i))\cup L_G$. Then
\[
|W|\leq |\mathbf{L}|\cdot 2c_0dm+\sum_{v\in \mathbf{M}}22d_H(v)m^9+2|\mathbf{S}|22m^4\cdot 4m+\tfrac{\varepsilon_2d}{2}\leq 30dm^{11}.
\]
For a given vertex set $Y$ and $i\in[2|\mathbf{S}|]$, we say a web $Z_i$ is \emph{$Y$-good} if $|\mathsf{Int}(Z_i)\cap Y|\le 11 m^{12}$ (which is at most half of the number of second-level branches in the web $Z_i$). 
To extend $f$ to $V(H)$ whilst finding the desired paths for the adjacencies in $H_1$, we use the following notion of \emph{good $\ell$-path system}.
We define $(X,I,I',\mathcal{Q},f)$ to be a \emph{good $\ell$-path system} if the following conditions hold:
\stepcounter{propcounter}
\begin{enumerate}[label = ({\bfseries \Alph{propcounter}\arabic{enumi}})]
\rm\item\label{gopa1} $X\subseteq \mathbf{S}$ and $f$ injectively maps the vertex set $X$ to the index set $I\<[2|\mathbf{S}|]$.
\rm\item\label{gopa2} $\mathcal{Q}$ is a collection of internally vertex-disjoint paths $Q_{x,y}$ of length $\ell$ for all edges $xy\in E(H_1)$ touching $X$, such that $Q_{x,y}$ is a $z_{f(x)},z_{f(y)}$-path disjoint from $W\setminus ( \mathsf{Int}(Z_{f(x)})\cup\mathsf{Int}(Z_{f(y)}))$.
\rm\item\label{gopa3} In particular, $Q_{x,y}$ begins (or ends) with a subpath $P_{x}(y)$ (resp. $P_{y}(x)$) within the object $Z_{f(x)}$ (resp. $Z_{f(y)}$) connecting the core $z_{f(x)}$ (resp. $z_{f(y)}$) to $\mathsf{Ext}(Z_{f(x)})$. Moreover, we write $Q_{x,y}'$ for the middle segment of $Q_{x,y}$, i.e. $Q_{x,y}'=Q_{x,y}\setminus (P_{x}(y)\cup P_{y}(x))$ and let $\mathcal{Q}'$ be the family of these paths $Q_{x,y}'$.
\rm\item\label{gopa6} $I'=\{i\in [2|\mathbf{S}|]:Z_{i} \ \text{is} \ \text{not} \ V(\mathcal{Q}')\text{-good}\}$ and $I'\cap I=\varnothing$.
\end{enumerate}

Now it suffices to build a good $\ell$-path system with $X=\mathbf{S}$. 
We proceed with our construction as follows.

\noindent
\textbf{Step $0$}. Fix an arbitrary ordering $\sigma$ on $\mathbf{S}$, say the first vertex is $x_1$. Let $X_1=\{x_1\}$, $f(x_1)=1$, $I_1=\{1\}$, $I'_1=\varnothing$ and $\mathcal{Q}_1=\varnothing$. Then by definition $(X_1,I_1,I'_1,\mathcal{Q}_1,f|_{X_{1}})$ is a good $\ell$-path system. Proceed to Step $1$.

\noindent
\textbf{Step $i$}. Stop if either $X_i=\mathbf{S}$ or $I_i\cup I'_i=[2|\mathbf{S}|]$,  otherwise we continue:
 \stepcounter{propcounter}
\begin{enumerate}[label = (\roman{enumi})]
       \item\label{youtiao1} Let $x$ be the first vertex in $\sigma$ on $\mathbf{S}\backslash X_i$. Choose a $V(\mathcal{Q}'_i)$-good object $Z_{t}$ with $t\in[2|\mathbf{S}|]\backslash (I_i\cup I'_{i})$ and define $f(x)=t$.
 \item\label{youtiao2} Find internally vertex-disjoint paths $Q_{x,y}$ for every neighbor $y$ of $x$ in $X_i\cup \mathbf{M}\cup \mathbf{L}$ satisfying \ref{gopa2}-\ref{gopa3}. 
     Once this is done, we add these paths to $\mathcal{Q}_i$ to get $\mathcal{Q}_{i+1}$.

  \item\label{youtiao3}  Update bad webs by setting $I'_{i+1}=\{i'\in[2|\mathbf{S}|]: Z_{i'} \ \text{is} \ \text{not} \ V(\mathcal{Q}'_{i+1})\text{-} \text{good}\}$, and define 
 $I_{i+1}=(I_i\cup\{t\})\backslash I'_{i+1}$, $X_{i+1}=f^{-1}(I_{i+1})$ and replace $f$ with its restriction $f|_{X_{i+1}}$.
  \item Proceed to \textbf{Step} $(i+1)$ with a good $\ell$-path system $(X_{i+1},I_{i+1},I'_{i+1},\mathcal{Q}_{i+1},$ $f|_{X_{i+1}})$.
\end{enumerate}
Now we claim the following result and postpone its proof temporarily.
\begin{claim}\label{cl1}
  In each step the desired paths in \emph{\ref{youtiao2}} can be successfully found.
\end{claim}
Claim~\ref{cl1} implies that $|I_i\cup I'_i|$ is strictly increasing at each step and the above process must terminate in at most $2|\mathbf{S}|$ steps. Let $(X,I,I',\mathcal{Q},f)$ be the final good $\ell$-path system returned from the above process and $\mathcal{Q}'$ be given as in \ref{gopa3}. Note that the sequence $|X_1|,|X_2|,\ldots$ might not be an increasing sequence, as we may delete some elements when updating the list of bad webs. Next we show that the process must terminate with $X=\mathbf{S}$.

Observe that by the definition of $W$, for each $v\in\mathbf{M}\cup \mathbf{L}$, $Z_v$ is $V(\mathcal{Q}')$-good, and $\mathcal{Q}'$ might contain some paths whose vertex set intersects $\mathsf{Int}(Z_{i'})\setminus \mathsf{Ctr}(Z_{i'})$ with $i'\in I'$. As at most $m^4$ paths are added at each step \ref{youtiao2}, we have $|I'|\leq \frac{2|\mathbf{S}|m^4\cdot m^3}{11m^{12}}=\frac{2|\mathbf{S}|}{11m^{5}}<|\mathbf{S}|$. Thus, $|I\cup I'|< 2|\mathbf{S}|$, and then the process terminates with $X=\mathbf{S}$. To complete the proof, it remains to show that all connections in \ref{youtiao2} can be guaranteed in each step.

\begin{proof}[Proof of Claim \ref{cl1}]\renewcommand*{\qedsymbol}{$\blacksquare$}
Given a good $\ell$-path system $(X_i,I_i,I'_i,\mathcal{Q}_i,f|_{X_{i}})$ and $x\in \mathbf{S}\setminus X_i$, $Z_{f(x)}$ as in \ref{youtiao1}, we let $\{y_1,\ldots,y_s\}=N_{H_1}(x)\cap (X_i\cup \mathbf{M}\cup \mathbf{L})$ and recall that our aim is to build pairwise internally disjoint paths $Q_{x,y_j}$ for all $j\in[s]$, each being a $z_{f(x)},z_{f(y_j)}$-path of length $\ell$. Note that by definition $Z_{f(y_j)}$ is $V(\mathcal{Q}'_i)$-good as $y_j\in X_i\cup \mathbf{M}\cup \mathbf{L}$ for every $j\in [s]$. Recall that $Z_{f(x)}$ is actually a $(22m^4,m^{8},\frac{d}{20},4m)$-web that is also $V(\mathcal{Q}'_i)$-good by our choice.

Fix a vertex $y=y_j$ as above and set $D=10^{-7}dm^{12}$ and $W'=W\cup \mathsf{Int}(Z_{f(x)})\cup \mathsf{Int}(Z_{f(y)})\cup V(\mathcal{Q}_i)$. Thus $|W'|\leq 30dm^{11}+22m^{4}[4m+m^{8}\cdot 4m]+\max\{22\cdot dm^{-10}\cdot[4m+m^{8}\cdot 4m],c_0d\}+m^3|\mathcal{Q}_i|\leq \frac{1}{2}Dm^{-\frac{3}{4}}.$
Applying Lemma \ref{link} to $G$ with $y=12$ and $W'$ playing the role of $W$, we obtain a $(D,m,20m)$-adjuster in $G-W'$, denoted as $\mathcal{A}=(v_1,F_1,v_2,F_2,A)$. Then by the definition of adjuster, $|A|\leq 200m^2$ and $\ell(\mathcal{A})\leq |A|+1\leq 210m^2$.

Recall that $Z_{f(x)}$ and $Z_{f(y)}$ are $V(\mathcal{Q}'_i)$-good. we shall see that they still have large boundaries for further connections. Consider the case when $y\in \mathbf{S}\cup \mathbf{M}$, that is, $Z_{f(y)}$ is either a $(22m^4,m^{8},\frac{d}{20},4m)$-web or a $(22d_H(y),m^8, \frac{dm^4}{20d_H(y)},4m)$-web with $d_H(y)\ge m^4$. Here we may take the case $y\in \mathbf{M}$ for instance (the case $y\in \mathbf{S}$ is much easier).
Note that there are at most $d_H(x)$ branches of $Z_{f(y)}$ are used for previous connections in $\mathcal{Q}_i$. 
Hence, there are at least $21d_H(y)m^{8}-11m^{12}\ge 10d_H(y)m^{8}$ available paths in $\mathsf{Int}(Z_{f(y)})\backslash \mathsf{Ctr}(Z_{f(y)})$, that is, the second-level branches which are not touched by $\mathcal{Q}_i$. Let $U_y\subseteq \mathsf{Ext}(Z_{f(y)})$ be the union of the leaves of the pendant stars attached to the ends of these available paths. Then $|U_y|\geq \frac{1}{2}dm^{12}-dm^3\ge 4D$. Here, subtracted term $dm^3$ counts the vertices from the previously used paths, and each path has length $m^3$ and the number of paths is at most $e(H)<Ce(H)\leq d$.
Similarly, the case when $Z_{f(y)}$ is a $(c_0d,m^{13},2m)$-unit, also witnesses such a vertex set $U_y\subseteq \mathsf{Ext}(Z_{f(y)})$ of size at least $c_0dm^{13}-dm^3>4D$. Thus by taking subsets and renaming, there are two disjoint vertex sets $U_{x}\subseteq \mathsf{Ext}(Z_{f(x)})$ and $U_y\subseteq \mathsf{Ext}(Z_{f(y)})$ each of size $D$ and they are disjoint from $F_1$ and $F_2$.

Let $\ell'=\ell-34m-\ell(\mathcal{A})$. Since $d\geq \log^sn$, $|A\cup W'|\leq 200m^2+\frac{1}{2}Dm^{-\frac{3}{4}}\leq Dm^{-\frac{3}{4}}$, by applying Lemma \ref{extadju}, we obtain vertex-disjoint paths, say $P_{x}$ and $P_{y}$ with $\ell'\leq \ell(P_x)+\ell(P_y)\leq \ell'+18m$, and we may further assume that $P_{x}$ is a $v_1,v_x$-path and $P_y$ is a $v_2,v_y$-path for some $v_x\in U_x$, $v_y\in U_y$.
On the other hand, by extending $P_x$ (similarly $P_y$) within the object $Z_{f(x)}$ (resp. $Z_{f(y)}$), we can obtain a $v_1,z_{f(x)}$-path say $P_x'$ (resp. $P_y'$) of length at most $8m+\ell(P_x)$. Note that $\ell'\leq \ell(P_x')+\ell(P_y')\leq \ell'+34m$. Thus $\ell(\mathcal{A})\leq \ell-\ell(P_x')-\ell(P_y')\leq \ell(\mathcal{A})+34m$. As $\mathcal{A}$ is a $(D,m,20m)$-adjuster, there is a $v_1,v_2$-path $R$ (in $A$) of length $\ell-\ell(P_x')-\ell(P_y')$, which together with $P_x',P_y'$ yields a $z_{f(x)},z_{f(y)}$-path of length $\ell$ as desired, which is denoted as $Q_{x,y}$. We can greedily build the pairwise disjoint paths $Q_{x,y}$ for all $y\in \{y_1,\ldots,y_s\}$ using the same argument as above.
\end{proof}

\noindent\textbf{Second round: Finding the desired paths (in $G$) for the adjacencies in $H_2$.}\medskip

Let $\mathcal{Q}$ be the resulting family of paths for the adjacencies in $H_1$ and $f$ be the resulting embedding of $V(H)$ returned from the first round. Note that $|\mathcal{Q}|\leq \ell\cdot e(H_1)\leq\ell\cdot e(H)< dm^3$ and $\mathcal{Q}$ is disjoint from
$\bigcup_{v\in \mathbf{L}\cup \mathbf{M}}\mathsf{Int}(Z_v)$. Recall that
$|\mathsf{Ext}(Z_v)|=\frac{11dm^{12}}{10}$ for each $v\in \mathbf{M}$ and
$|\mathsf{Ext}(Z_v)|=c_0dm^{13}$ for each $v\in \mathbf{L}$. Update $W=(\bigcup_{v\in \mathbf{L}\cup\mathbf{M}} \mathsf{Int}(Z_v))\cup V(\mathcal{Q})$. Then
\[
|W|\leq |\mathbf{L}|\cdot 2c_0dm+\sum_{v\in \mathbf{M}}22d_H(v)m^9+dm^3\leq 30dm^{11}.
\]

Observe that every $v\in \mathbf{L}$ witnesses at least $c_0d-d_H(v)$ available branches in the unit $Z_{f(v)}$ and every $v\in \mathbf{M}$ witnesses at least $22d_H(v)-d_H(v)$ branches in $\mathsf{Ctr}(Z_{f(v)})$, which are disjoint from $V(\mathcal{Q})$. Similarly, for each $x\in \mathbf{L}\cup\mathbf{M}$, let $V_x\subseteq \mathsf{Ext}(Z_x)$ be the union of the leaves from the pendant stars attached to one end of these available paths. Then $|V_x|\geq \min\{(c_0d-d_H(v))m^{13}-dm^3,\frac{21dm^{12}}{10}-dm^3\}\geq dm^{12}$.

Let $I\<E(H_2)$ be a maximum set of edges for which there exists a collection $\mathcal{P}_I=\{P_e:e\in I\}$ of internally vertex-disjoint paths under the following rules.
\stepcounter{propcounter}
\begin{enumerate}[label = ({\bfseries \Alph{propcounter}\arabic{enumi}})]
\rm\item\label{pathlab1} For each $xy=e\in E(H_2)$, $P_e$ is a $z_{f(x)},z_{f(y)}$-path of length $\ell$ and $P_e$ is disjoint from $W\setminus ( \mathsf{Int}(Z_{f(x)})\cup\mathsf{Int}(Z_{f(y)}))$.
\rm\item\label{pathlab2} $P_{e}$ begins (or ends) with the unique subpath within the object $Z_{f(x)}$ (resp. $Z_{f(y)}$) connecting the core vertex $z_{f(x)}$ (resp. $z_{f(y)}$) and some vertex in $\mathsf{Ext}(Z_{f(x)})$.
\end{enumerate}
\begin{claim}\label{c47}
$I=E(H_2)$.
\end{claim}
\begin{proof}[Proof of Claim \ref{c47}]\renewcommand*{\qedsymbol}{$\blacksquare$}
Suppose to the contrary that there exists $e=x_1x_2\in E(H_2)\setminus I$ with no desired path in $\mathcal{P}_I$ between their corresponding objects, say $Z_1,Z_2$. Set $D=10^{-7}dm^{12}$ and $W'=W\cup V(\mathcal{P}_I)$, and thus \[|W'|\leq 30dm^{11}+\ell e(H)\leq 32dm^{11}<\tfrac{1}{2}Dm^{-\tfrac{3}{4}}\leq\tfrac{\rho(2D)2D}{4}.\] Applying Lemma \ref{link} with $y=12$ and $W'$ playing the role of $W$, we obtain a $(D,m,20m)$-adjuster $\mathcal{A}=(v_1,F_1,v_2,F_2,A)$ in $G-W'$. Note that $|A|\leq 200m^2$ and $\ell(\mathcal{A})\leq |A|+1\leq 210m^2$. Now let $\ell'=\ell-30m-\ell(\mathcal{A})$. On the other hand, as $|V_{x_i}|\geq dm^{12}\geq 2D$ for $i\in[2]$, there are disjoint vertex sets $U_1\subseteq V_{x_1}$ and $U_2\subseteq V_{x_2}$ of size $D$. Since $d\geq \log^sn$, $|A\cup W'|\leq 200m^2+\frac{1}{2}Dm^{-\frac{3}{4}}\leq Dm^{-\frac{3}{4}}$, applying Lemma \ref{extadju} gives vertex-disjoint paths say $Q_{1}$, $Q_{2}$ with $\ell'\leq \ell(Q_{1})+\ell(Q_{2})\leq \ell'+18m$ and we may assume that $Q_{1}$ is a $u_{1},v_{1}$-path and $Q_{2}$ is a $u_{2},v_{2}$-path for some $u_{1}\in U_1, u_{2}\in U_2$.
By the adjustment as above via $\mathcal{A}$, we can easily extend $Q_{1},Q_{2}$ into a desired path of length $\ell$ connecting $z_1$ and $z_2$ while avoiding $W'$, denoted as $P_e$. Thus $\{P_{e}\}\cup \mathcal{P}_I$ yields a contradiction to the maximality of $\mathcal{P}_I$.
\end{proof}

In summary, the resulting families of paths $\mathcal{Q}$ in the first round and $\mathcal{P}_{E(H_2)}$ in the second round form a copy of $TH^{(\ell)}$ as desired.
\end{proof}

\section{Proof of Lemma \ref{dense2-3}}\label{sec:log-dense}
To prove \ref{deenu1}, take $\beta_1=\frac{\beta}{2}$ and choose additional constants $M, \varepsilon_1$ such that
$$\frac{1}{h}\ll \frac{1}{K}, \alpha, c\ll \frac{1}{M}\ll\varepsilon_1\ll \beta,\varepsilon.$$
First, we apply Lemma \ref{regudegver} to obtain an $\varepsilon_1$-regular partition $\mathcal{P}=\{V_0,V_1,\ldots,V_k\}$ of $V(G)$ ($k\leq M$). Arbitrarily choose an $(\varepsilon_1,\beta_1)$-regular pair, say $(V_1,V_2)$. Note that $|V_1|=|V_2|\geq\frac{(1-\varepsilon_1)n}{k}$. For each $i\in[2]$, a vertex $v$ in $V_i$ is \emph{bad} if $d_{G[V_{3-i}]}(v)<(\beta_1-\varepsilon_1)|V_{3-i}|$, and denote by $B_i$ the set of bad vertices in $V_i$. By Lemma \ref{regular}, $|B_i|\leq \varepsilon_1|V_i|$. As $H$ is $(\alpha e(H),c\log h)$-biseparable, there exists $E_1\subseteq E(H)$ with $|E_1|\leq \alpha e(H)$ such that each component of $H\setminus E_1$ is a bipartite graph on at most $c\log h$ vertices. Let $C_1,C_2,\ldots C_m$ be all components of $H\setminus E_1$. Note that $m\geq \frac{h}{c\log h}$. Now we shall embed each $C_i$ in $U:=(V_1\cup V_2)\backslash(B_1\cup B_2)$.

For simplicity, let $\rho n=|V_1\cup V_2|$ and $t=c\log h$. Then $\rho\geq\frac{2(1-\varepsilon_1)}{k}$, $|B_1\cup B_2|\leq \varepsilon_1\rho n$. Suppose that $C_1,\ldots,C_{i-1}$ has been embedded in $U$, and $C_i$ is the current component to be embedded. Observe that since $d(H)\geq K$,
$$\sum_{j=1}^{i-1}|C_j|\leq h\leq \tfrac{2\beta n}{\varepsilon d(H)}\leq \tfrac{\rho n}{3}.$$
Let $R_j=V_j\cap \left(\bigcup_{z=1}^{i-1}C_z\right)$ for each $j\in[2]$. Now we shall embed $C_i$ into $(V_1\cup V_2)\backslash (R_1\cup R_2\cup B_1\cup B_2)$, which has size $\xi n$, where $\xi\geq (1-\varepsilon_1-\frac{1}{3})\rho\geq \frac{\rho}{2}$.
Since $t=c\log h$ and $c\ll\frac{1}{k}, \beta$, it follows from Lemma \ref{exKST} that 
$$\mathrm{ex}(\xi n,K_{t,t})\leq t^{\frac{1}{t}}(\xi n)^{2-\frac{1}{t}}=\left(\tfrac{t}{\xi n}\right)^{\frac{1}{t}} (\xi n)^{2}<\tfrac{(\beta_1-\varepsilon_1)\cdot(\xi n)^2}{100}\leq e(U).$$
Hence we can embed a copy of $K_{t,t}\supseteq C_i$ into $U$. Let $V(H)=\{v_1,v_2,\ldots, v_h\}$. Denote by $\varphi:V(H)\rightarrow V(G)$ the resulting embedding of $C_1,\ldots,C_m$ and let $\varphi(v_i)=u_i$ for each $i\in[h]$.

Next, for all edges in $H[C_i,C_j]$ $(i,j\in[m])$, we embed pairwise internally disjoint paths of length at most $4$ avoiding $\varphi(V(H))$ in $G[V_1,V_2]$. Suppose that $v_iv_j$ $(v_i\in C_i, v_j\in C_j)$ is the current edge for which we shall find a $u_i,u_j$-path whilst avoiding all internal vertices used in previous connections. Denote by $W$ the set of all internal vertices used in previous connections, then $|W|\leq 3|E_1|$. Recall that $|E_1|\leq \alpha e(H)\le \frac{\alpha\beta n}{\varepsilon }$. If $u_i,u_j$ are located in the same part, say $u_i,u_j\in V_1$, then we have that for each $p\in\{i,j\}$
\begin{equation*}
|N(u_p)\cap (V_2\backslash W)|\geq (\beta_1-\varepsilon_1)|V_2|-3\alpha e(H)\geq \tfrac{\beta_1|V_2|}{2}.
\end{equation*}
Thus, fixing an arbitrary $A\subseteq V_1$ such that $A\cap (\varphi(V(H))\cup W)=\varnothing$ and $|A|\geq \varepsilon_1|V_1|$, we have that for each $p\in\{i,j\}$
\begin{equation*}
|d(A,N(u_p)\cap (V_2\backslash W))|>\beta_1-\varepsilon_1.
\end{equation*}
Therefore, any typical vertex $a\in A$ with positive degree to $N(u_p)\cap (V_2\backslash W)$, $p\in\{i,j\}$, yields a $u_i,u_j$-path of length $4$ as desired. We can choose such typical vertices by Lemma~\ref{regular}.
The case when $u_i,u_j$ are in different parts is simpler. Say $u_i\in V_1$, $u_j\in V_2$, then it is easy to find a $u_i,u_j$-path of length $3$ using edges between $N(u_i)\cap (V_2\backslash W)$ and $N(u_j)\cap (V_1\backslash W)$; we omit the details. Note that all these paths corresponding to $E_1=\cup H[C_i,C_j]$ together with $C_1,\ldots,C_m$ form a desired copy of $TH^{(\le 3)}$.

\medskip

To prove \ref{deenu2}, we first need the following claim.
\begin{claim}\label{sepa}
If $F$ is $\kappa$-degenerate and $(\alpha e(F),\log f)$-biseparable, then $H=F^{\Box r}$ is $r\kappa$-degenerate and $(\alpha e(H),\log^r f)$-biseparable.
\end{claim}
\begin{proof}[Proof of Claim \ref{sepa}]\renewcommand*{\qedsymbol}{$\blacksquare$}
Since $F$ is $\kappa$-degenerate, there exists an ordering of vertices in $F$, say $v_1,\ldots,v_f$ such that each $v_i$ has at most $\kappa$ neighbors in $\{v_{i+1},\ldots,v_f\}$. Given two vertices $\boldsymbol{x}=(x_1,x_2,\ldots,x_r)$, $\boldsymbol{y}=(y_1,y_2,\ldots,y_r)$ in $V(H)$, we define an ordering on $V(H)$ by letting $\boldsymbol{x}<\boldsymbol{y}$ if there exists $i\in[r]$ such that $x_j=y_j$ for all $j\in[i]$ but $x_{i+1}\neq y_{i+1}$, say $x_{i+1}=v_{k_1}$, $y_{i+1}=v_{k_2}$ for some $k_1<k_2$. It is obvious that the resulting ordering, say $\boldsymbol{w_1}, \boldsymbol{w_2}, \ldots, \boldsymbol{w_h}$, satisfies that every vertex $\boldsymbol{w_i}$ has at most $r\kappa$ neighbors in $\{\boldsymbol{w_{i+1}},\ldots,\boldsymbol{w_h}\}$. Thus, $H$ is $r\kappa$-degenerate.

Next, as $F$ is $(\alpha e(F),\log f)$-biseparable, we get that there exists $E_1\subseteq E(F)$ with $|E_1|\leq \alpha e(F)$ such that each component of $F\setminus E_1$ is bipartite on at most $\log f$ vertices. Let $E_2$ be the set of all edges $\boldsymbol{x}\boldsymbol{y} \in E(H)$ by writing $\boldsymbol{x}=(x_1,x_2,\ldots,x_r)$ and $\boldsymbol{y}=(y_1,y_2,\ldots,y_r)$ such that $x_iy_i\in E_1$ for some $i\in[r]$.
Hence, $|E_2|=rf^{r-1}|E_1|\leq \alpha rf^{r-1}e(F)=\alpha e(H)$. By the construction of $H$, it is easy to see that each component of $H\setminus E_2$ is bipartite on at most $\log^r f$ vertices, as claimed.
\end{proof}
Now we choose an additional constant $M$ such that $\frac{1}{r}, \frac{1}{f} \ll\frac{1}{K},\alpha\ll \frac{1}{k}\ll \varepsilon_1\ll \beta,  \varepsilon, \frac{1}{\kappa}$, and again apply Lemma \ref{regudegver} with $\beta_1=\frac{\beta}{2}$ to obtain an $\varepsilon_1$-regular partition $\mathcal{P}=\{V_0,V_1,...V_k\}$ ($k\leq M$) of $V(G)$. Arbitrarily choose an $(\varepsilon_1,\beta_1)$-regular pair, say $(V_1,V_2)$. Note that $|V_1|=|V_2|\geq\frac{(1-\varepsilon_1)n}{k}$. We shall embed all components $C_1,\ldots,C_m$ $(m\geq\frac{f^r}{\log^rf})$ of $H\setminus E_2$ into $V_1\cup V_2$ one by one disjointly using Lemma \ref{Leethm}. To this end, we need to check two inequalities mentioned in Lemma \ref{Leethm}, that is,
\begin{equation}\label{1ineq}
\log^rf\geq \left(\tfrac{\beta}{2}\right)^{-K_{\ref{Leethm}}(r\kappa)^2}
\end{equation}
and
\begin{equation}\label{2ineq}
\tfrac{|V_1|}{2}\ge \tfrac{(1-\varepsilon_1)n}{2k}\geq\left(\tfrac{\beta}{2}\right)^{-K_{\ref{Leethm}}r\kappa}\log^rf.
\end{equation}
Under the condition $\log f>e^{K\kappa^2r}$, we obtain that the inequality $(\ref{1ineq})$ holds by taking $K\geq K_{\ref{Leethm}}\log \frac{2}{\beta}$.
As $n\geq d\geq \varepsilon e(H)=\frac{1}{2}\varepsilon rf^rd(F)$ and $\log f > \max\{e^{K\kappa^2 r}, K\log \log f\}$, we have that
\begin{equation*}
\log n\geq r\log f\geq Kr(\kappa+\log\log f)
\end{equation*}
holds, which implies inequality (\ref{2ineq}) from the choice $\frac{1}{K}\ll \frac{1}{k},\beta$.

For all edges in $H[C_i,C_j]$ $(i,j\in[m])$, we embed pairwise disjoint paths of length at most 4 in $G[V_1,V_2]$ using the same argument as in Part (1), which together with all $C_1,\ldots,C_m$ form the desired $TH^{(\le 3)}$.

\section{Proof of Lemma \ref{inte}}\label{sec:log-sparse}

The proof of Lemma \ref{inte} follows the same embedding strategy as in Lemma \ref{dense} and is easier since no adjusters are needed. The main difference is that we need the following lemma in place of Lemma~\ref{aver}, which reduces our constructions of $H$-subdivisions to locally dense graphs. 
We include a detailed proof in the Appendix \ref{appen-lem:3.6}.

\begin{lemma}\label{dense34}
Suppose $\frac{1}{h},\frac{1}{f}, \frac{1}{r}\ll \frac{1}{K},\alpha, c\ll\frac{1}{x},\frac{1}{\kappa},\varepsilon<1$ and $s,n,d\in \mathbb{N}$ satisfy $s\geq 1600$ and $\log^s n \leq d \leq \frac{n}{K}$. Let $H$ be an $h$-vertex graph and $G$ be a bipartite graph with $\delta(G)\geq d\geq \varepsilon e(H)$ satisfying one of the following conditions.
\begin{enumerate}[label = (\arabic{enumi})]
\rm \item \label{lem6-1-1} $H$ is $(\alpha e(H), c\log h)$-biseparable with $d(H)\geq K$.
\rm \item \label{lem6-1-2} $H=F^{\Box r}$, where $F$ is an $f$-vertex $\kappa$-degenerate $(\alpha e(F), \log f)$-biseparable graph with $d(F)\geq 1$ and $\log f > e^{K\kappa^2 r}$.
\end{enumerate}
Then either $G$ contains a $TH^{(\leq 7)}$ or $G$ is locally-$(dm^x, \frac{d}{2})$-dense for $m=\log^4\frac{n}{d}$.
\end{lemma}
\begin{proof}
Suppose that $G$ is not locally-$(dm^x, \frac{d}{2})$-dense, then there exists some $W\subseteq V(G)$ with $|W|\leq dm^x$ such that $d(G-W)<\frac{d}{2}$. As $\delta(G)\geq d$, we may assume $|W|>\frac{d}{2}$. Denote by $n_1:=|W|\leq dm^x$ and $n_2:=|V(G-W)|\geq n-dm^x$. Then
\begin{equation*}
\pi:=\tfrac{e(W,V(G-W))}{n_1n_2}>\tfrac{n_2\cdot \tfrac{d}{2}}{n_1n_2}= \tfrac{d}{2n_1}\geq\tfrac{1}{2m^x}.
\end{equation*}
Let $w\in V(G-W)$ be a vertex chosen uniformly at random, and let $A$ denote the set of neighbors of $w$ in $W$, and $X=|A|$. Then $\mathbb{E}[X]=\pi n_1>\frac{d}{2}$.

Let $Y$ be the random variable counting the number of pairs in $A$ with fewer than $4e(H)$ common neighbors in $G-W$. Then $\mathbb{E}[Y]\leq \frac{4e(H)}{n_2}\binom{n_1}{2}\leq \frac{2e(H)n_1^2}{n_2}.$ Using linearity of expectation, we obtain
\begin{equation*}
\mathbb{E}\left[X^2-\frac{\mathbb{E}[X]^2}{2\mathbb{E}[Y]}Y-\frac{\mathbb{E}[X]^2}{2}\right]\geq0.
\end{equation*}
Hence, there is a choice of $w$ such that this expression is nonnegative. Then
$$X^2\geq\frac{1}{2}\mathbb{E}[X]^2> \frac{d^2}{8} \quad \text{and}\quad
Y\leq 2\frac{X^2}{\mathbb{E}[X]^2}\mathbb{E}[Y]<\frac{4e(H)X^2}{\pi^2n_2}<\frac{16m^{2x}e(H)|A|^2}{n-dm^x}\leq \frac{|A|^2}{8},$$
where the last inequality holds as $n>Kd$ and thus $n-dm^x\geq 128m^{2x}e(H)$. Therefore,  $|A|=X>\frac{d}{4}$.

Define a graph $G_1=(V(G_1),E(G_1))$ with $V(G_1)=A$, and $uv$ is an edge of $G_1$ if and only if  $d_{G-W}(u,v)\geq 4e(H)$. Thus,\begin{equation*}
e(G_1)\geq \binom{|A|}{2}-Y\geq \frac{|A|^2}{4},
\end{equation*}
and
\begin{equation*}
d(G_1)\geq \tfrac{2e(G_1)}{|A|}\geq \tfrac{|A|}{2}\geq \tfrac{d}{8}\geq \tfrac{\varepsilon}{8}e(H).
\end{equation*}
In each of two cases \ref{lem6-1-1} and \ref{lem6-1-2}, applying Lemma \ref{dense2-3}~\ref{deenu1} and \ref{deenu2} accordingly to $G_1$ with $\beta=\frac{1}{2}$, we get a $TH^{(\leq3)}$ in $G_1$, denoted as $Q$. Now we shall replace each edge of $Q$ with a copy of $P_3$ in $G$. Let $V(Q)=\{u_1,u_2,\ldots,u_{t}\}$, and let $f:V(Q)\rightarrow V(G)$ be any injective mapping. Suppose $u_iu_j$ is the current edge for which we shall find a $f(u_i),f(u_j)$-path of length $2$ whilst avoiding all internal vertices used in previous connections. Since there are at most $2e(H)$ vertices in $N_{G-W}(u_i)\cap N_{G-W}(u_j)$ used in previous connections, there exists an un-used common neighbor $u_{ij}$ of $u_i$ and $u_j$, which forms a copy of $P_3$ in $G$. Thus we can find a $TH^{(\leq 7)}$ in $G$.
\end{proof}

\vspace{0.5cm}
\textbf{Acknowledgement.}
We thank the anonymous referees for their careful reading and valuable comments which greatly improve the presentation of our manuscript.

\begin{appendix}
\section{Proof of Proposition \ref{proball} and Lemma \ref{ball}}\label{app1}
\begin{proof}[Proof of Proposition \ref{proball}]
Suppose to the contrary that $|B_G^m(v)|<\frac{n}{2}$. Observe that $|B_G^1(v)|\geq d(v)\geq\varepsilon_2d$. By the expansion property, we have
\begin{equation*}
|B_G^i(v)|\geq |B_G^{i-1}(v)|\big(1+\rho(|B_G^{i-1}(v)|)\big),
\end{equation*}
whence
\begin{equation*}
\tfrac{n}{2}>|B_G^m(v)|\geq |B_G^1(v)|\prod_{j=1}^{m-1}\big(1+\rho(|B_G^j(v)|)\big)\geq |B_G^1(v)|\left(1+\tfrac{\rho(n)}{2}\right)^{m-1}.
\end{equation*}
Then, as $\log(1+\rho(n))\geq \tfrac{\rho(n)}{2}$, we obtain
\begin{equation*}
m\leq\tfrac{\log(\frac{n}{\varepsilon_2d})}{\log(1+\rho(n))}+1<\log^4\tfrac{n}{d},
\end{equation*}
which contradicts the choice of $m$ (recall that $m$ is the smallest even integer which is larger than $\log^4\frac{n}{d}$).
\end{proof}

\begin{proof}[Proof of Lemma \ref{ball}]
Given $\varepsilon_1,\varepsilon_2,s,x$ such that $s\geq 8x$, we are given constants satisfying $\tfrac{1}{n}, \tfrac{1}{d}\ll c,\frac{1}{K}\ll \varepsilon_1,\varepsilon_2,\tfrac{1}{x},\tfrac{1}{s}<\tfrac{1}{5}$. Let $H$ be an $n$-vertex $(\varepsilon_1,\varepsilon_2d)$-expander with $\delta(H)\geq \frac{d}{2}$, and $P_1,\ldots,P_t$ be consecutive shortest paths from $v$ in $B_H^m(v)$. Let $F=H-(U\backslash\{v\})$. 
\begin{claim}\label{apclaba}
The following properties hold.
    \begin{enumerate}
        \item [$(1)$] If $|B_{F}^p(v)|\leq dm^x$ and $p<m$, then
\begin{equation}\label{inequality}
|N_F(B_{F}^p(v))|\geq \tfrac{1}{2}|B_{F}^p(v)|\cdot \rho(|B_{F}^p(v)|).
\end{equation}
\item [$(2)$] $|B^1_{F}(v)|\geq \frac{d}{10}$.
    \end{enumerate}
\end{claim}

If Claim \ref{apclaba} holds and  $|B_{F}^p(v)|\leq dm^x$, then for each $1\leq p<m$, we have
\begin{align}
|N_F(B_{F}^p(v))|&\geq \tfrac{1}{2}|B_{F}^p(v)|\cdot\rho(|B_{F}^p(v)|)=\tfrac{\varepsilon_1|B_{F}^p(v)|}{2\log^2\left(\frac{15|B_{F}^p(v)|}{\varepsilon_2d}\right)}\nonumber\\
&\geq|B_{F}^p(v)|\cdot\tfrac{\varepsilon_1}{2\log^2\left(\frac{15dm^x}{\varepsilon_2d}\right)}
\geq\tfrac{|B_{F}^p(v)|}{\log^3(m^x)},\nonumber
\end{align}
where we have used that $|B_{F}^p(v)|\geq |B_{F}^1(v)|\geq \frac{d}{10}> \frac{\varepsilon_2d}{2}$ to apply the expansion property. Hence, we have
\begin{align}
|B_{F}^m(v)|&>\left(1+\tfrac{1}{\log^3(m^x)}\right)^{m-1}|B^1_{F}(v)|\geq \tfrac{d}{10}\left(1+\tfrac{1}{\log^3(m^x)}\right)^{m-1}\nonumber \\
&\geq \tfrac{d}{10}e^{\frac{m-1}{2\log^3(m^x)}}>de^{\log(m^x)}=dm^x,\nonumber
\end{align}
where the penultimate inequality holds as $1+y\geq e^{\tfrac{y}{2}}$ for any $0<y<1$, and the last inequality follows as $\frac{n}{d}$ and also $m$ are sufficiently large. Thus, we complete the proof. It remains to prove \ref{apclaba}.

\begin{proof}[Proof of Claim \ref{apclaba}]\renewcommand*{\qedsymbol}{$\blacksquare$}
    As the paths $P_i$ are consecutive shortest paths from $v$ in $B^m_H(v)$, only the first $p+2$ vertices of each path $P_i$, including $v$, can belong to $N_H(B_{H-\cup_{j<i}(V(P_j)\backslash\{v\})}^p(v))$. Hence, if $p<m$, then only the first $p+2$ vertices of each of the path $P_i$, including the vertex $v$, can belong to $N_H(B_{F}^p(v))$. On the other hand, as we have at most $cd$ paths $P_i$, if $p<m$, then $|N_H(B_{F}^p(v))\cap (U\backslash\{v\})|\leq(p+1)cd$, so that
\begin{equation}\label{ballineq}
|N_{H-F}(B_{F}^p(v))|\leq(p+1)cd.
\end{equation}

In particular, when $p=0$, the inequality (\ref{ballineq}) implies that $|N_{F}(v)|\geq |N_H(v)|-cd\geq \delta(H)-cd\geq \frac{d}{2}-cd\geq \frac{d}{10}$. Hence, $|B^1_{F}(v)|\geq |N_F(v)|\geq \frac{d}{10}$.

Next we aim to prove (\ref{inequality}). When $p=1$, by (\ref{ballineq}), we have
\begin{equation}\label{less1}
|N_F(B_{F}^1(v))|\geq |N_H(B_{F}^1(v))|-2cd.
\end{equation}
However, by the choice of $c$, $2cd\leq \frac{1}{2}\rho(\frac{d}{10})\cdot\frac{d}{10}\leq \frac{1}{2}\rho(|B^1_{F}(v)|)\cdot |B^1_{F}(v)|$. Thus, (\ref{less1}) becomes
\begin{align}
N_F(B_{F}^1(v))&\geq |N_H(B_{F}^1(v))|-\tfrac{1}{2}\rho(|B^1_{F}(v)|)\cdot |B^1_{F}(v)| \nonumber \\
&>\tfrac{1}{2}\rho(|B^1_{F}(v)|)\cdot |B^1_{F}(v)|,\nonumber
\end{align}
where the last inequality holds because $|N_H(B_{F}^1(v))|>\rho(|B^1_{F}(v)|)\cdot |B^1_{F}(v)|$.

Now, let $p\geq 2$. Suppose that (\ref{inequality}) holds for all $1\leq p'<p$. By (\ref{ballineq}), it remains to prove that
\begin{equation}\label{less2}
(p+1)cd\leq \tfrac{1}{2}\rho(|B_{F}^p(v)|)\cdot |B_{F}^p(v)|.
\end{equation}
Let $\alpha$ be defined by $|B_{F}^p(v)|=\frac{\alpha\varepsilon_2d}{15}$ and note that $\alpha\geq 3$. Then $\rho(|B_{F}^p(v)|)=\frac{\varepsilon_1}{\log^2\alpha}$. By the induction hypothesis, we have
\begin{equation*}
\left(1+\tfrac{\varepsilon_1}{2\log^2\alpha}\right)^{p-1}\leq\tfrac{|B_{F}^p(v)|}{|B^1_{F}(v)|}\leq \tfrac{\alpha\varepsilon_2d}{15}\cdot\tfrac{10}{d}=\tfrac{2}{3}\alpha\varepsilon_2<\alpha.
\end{equation*}
Thus,
\begin{equation*}
p-1\leq\tfrac{\log \alpha}{\log\left(1+\frac{\varepsilon_1}{2\log^2\alpha}\right)}\leq\tfrac{\log \alpha}{\frac{1}{2}\cdot\frac{\varepsilon_1}{2\log^2\alpha}}=\tfrac{4\log^3\alpha}{\varepsilon_1},
\end{equation*}
where the last inequality holds as $\log(1+x)\geq \frac{x}{2}$ for all $0<x<1$. Note that when $\alpha\geq 3$, $\frac{\log^5\alpha}{\alpha}$ is bounded by some universal constant, say $L$. Therefore
\begin{align}
(p+1)cd&\leq \tfrac{12\log^3\alpha}{\varepsilon_1}\cdot cd\leq \tfrac{12cd}{\varepsilon_1}\cdot \tfrac{L\alpha}{\log^2\alpha}=\tfrac{180cL}{\varepsilon_1\varepsilon_2}\cdot\tfrac{\alpha\varepsilon_2d}{15\log^2\alpha} \nonumber \\
&=\tfrac{180cL}{\varepsilon_1^2\varepsilon_2}\cdot\rho(|B_{F}^p(v)|)\cdot |B_{F}^p(v)|\leq\tfrac{1}{2}\rho(|B_{F}^p(v)|)\cdot |B_{F}^p(v)|, \nonumber
\end{align}
for $c$ sufficiently small depending on $\varepsilon_1,\varepsilon_2$ and $L$, and so the inequality (\ref{less2}) holds.
\end{proof}

\end{proof}

\section{Proof of Lemma \ref{middlewebs}: finding webs}\label{appeBwebs}
\begin{proof}[Proof of Lemma \ref{middlewebs}]
Recall that $G$ is a $(dm^x,\frac{d}{2})$-dense bipartite $(\varepsilon_1,\varepsilon_2d)$-expander, and $W\subseteq V(G)$ with $|W|\leq 100dm^{x-2y+z-4}$.
We first prove that the following holds.
\begin{claim}\label{lem-a}
For any set $X$ of size at most $dm^{x}$, the graph $G-X$ contains a star $S$ with at least $\frac{d}{4}$ leaves. In particular, the center vertex of $S$ lies in $V_1$.
\end{claim}

\begin{proof}[Proof of Claim \ref{lem-a}]\renewcommand*{\qedsymbol}{$\blacksquare$}
By the assumption that $G$ is $(dm^x,\frac{d}{2})$-dense, we have $d(G-X)\geq \frac{d}{2}$. Let $V'_1=V_1\backslash X$ and $V'_2=V_2\backslash X$. Since
\begin{equation*}
\frac{\sum_{v\in V'_1}d(v)}{|V'_1|}=\frac{|E(G-X)|}{|V'_1|}>\frac{d(G-X)}{2}\ge \frac{d}{4}.
\end{equation*}
Hence, $G-X$ contains a star $S$ with $\frac{d}{4}$ leaves, whose center vertex lies in $V_1$.
\end{proof}

Recall that our main goal is to construct a web in $G-W$. We shall first build many vertex-disjoint units as follows.
\begin{claim}\label{lem-b}
The graph $G-W$ contains $100\gamma m^{x-2y+z-3}$ vertex-disjoint $(2m^{y-z},\frac{dm^z}{10\gamma},m+2)$-units.
\end{claim}
\begin{proof}[Proof of Claim \ref{lem-b}]\renewcommand*{\qedsymbol}{$\blacksquare$}
Suppose we have found a collection of units $F_1,\ldots,F_t$ as desired for some $0\leq t<100\gamma m^{x-2y+z-3}$. Then the set $X':=\bigcup_{i\in[t]}V(F_i)$ has size at most $21dm^{x-y+z-3}$ and $|X'\cup W|\leq 22dm^{x-1}<dm^x$. By Claim \ref{lem-a}, we can find vertex-disjoint stars $S_{1}, \ldots$, $S_{m^{x-y+z-2}}$, $T_{1}, \ldots,$ $T_{\gamma m^{x-z-1}}$ with centers $u_1, \ldots, u_{m^{x-y+z-2}}$, $v_1, \ldots, v_{\gamma m^{x-z-1}}$, respectively in $G-W-X'$ such that all centers lie in $V_1$ and each $S_i$ has exactly $\frac{d}{4}$ leaves and each $T_i$ has exactly $\frac{dm^z}{5\gamma}$ leaves. This can be done because $|\bigcup_{i\in [m^{x-y+z-2}]}V(S_i)|+|\bigcup_{i\in [\gamma m^{x-z-1}]}V(T_i)|+|W|+|X'|\leq dm^x$. For simplicity, set $Z=\{u_1, \ldots, u_{m^{x-y+z-2}},$ $v_1, \ldots, v_{\gamma m^{x-z-1}}\}$.

Let $\mathcal{P}$ be a maximum collection of internally disjoint paths $P_{ij}$ in $G-W-X'$ satisfying the following rules.
\stepcounter{propcounter}
\begin{enumerate}[label = ({\bfseries \Alph{propcounter}\arabic{enumi}})]
\rm\item\label{1pathlab2} Each path $P_{ij}$ in $\mathcal{P}$ is a unique $u_i,v_j$-path of length at most $m+2$.
\rm\item\label{1pathlab3} Each $P_{ij}$ does not contain any vertex in $Z$ as an internal vertex.
\end{enumerate}

Now we claim that there is a center $u_i$ connected to at least $2m^{y-z}$ distinct centers $v_j$ via the paths in $\mathcal{P}$. Suppose to the contrary that every $u_i$ is connected to less than $2m^{y-z}$ centers $v_j$. Then $|\mathcal{P}|\leq 2m^{x-2}$ and $|V(\mathcal{P})|\leq 2m^{y-z} \cdot(m+2)\cdot m^{x-y+z-2}\leq4 m^{x-1}$. Let
\begin{equation*}
U:=\left(\bigcup_{i\in [m^{x-y+z-2}]}(S_i\backslash\{u_i\})\right)\backslash V(\mathcal{P}),
\end{equation*}
and $V$ be the set of leaves of all stars $T_i$ whose centers are not used as endpoints of paths in $\mathcal{P}$. Then we have
\begin{equation}\label{webine1}
|U|\geq \frac{d}{4}\cdot m^{x-y+z-2}-4 m^{x-1}>\frac{dm^{x-y+z-2}}{10},
\end{equation}
and
\begin{equation}\label{webine2}
|V|\geq \frac{dm^z}{5\gamma}\cdot(\gamma m^{x-z-1}-2m^{x-2})\geq \frac{dm^z}{5\gamma}\cdot\frac{1}{2}\gamma m^{x-z-1}
= \frac{dm^{x-1}}{10}> \frac{dm^{x-y+z-2}}{10},
\end{equation}
where (\ref{webine2}) follows as $\gamma\geq m^z$ and $y>z$. On the other hand,
\begin{align}
\ &|W|+|X'|+|\mathsf{Int}(\mathcal{P})|+|Z| \nonumber \\
&\leq 100dm^{x-2y+z-4}+21dm^{x-y+z-3}+4 m^{x-1}+m^{x-y+z-2}+\gamma m^{x-z-1} \label{webine3}\\
&\leq \frac{1}{4}\rho\left(\frac{dm^{x-y+z-2}}{10}\right)\cdot \frac{dm^{x-y+z-2}}{10} \nonumber.
\end{align}
The last inequality in (\ref{webine3}) holds as $y<z+10$.
Hence, applying Lemma \ref{distance} with $U,V,W\cup X'\cup \mathsf{Int}(\mathcal{P})\cup Z$ playing the roles of $X_1,X_2,W$, respectively, we obtain vertices $x_{k_1}\in U, x_{k_2}\in V$ and a path of length at most $m$ connecting $x_{k_1}$ and $x_{k_2}$ whilst avoiding vertices in $W\cup X'\cup \mathsf{Int}(\mathcal{P})\cup Z$. Denote by $S_{k_1}, T_{k_2}$ the stars which contain $x_{k_1}, x_{k_2}$ as leaves, respectively. This yields a $u_{k_1},v_{k_2}$-path $P_{k_1,k_2}$, which is internally disjoint from $W\cup X'\cup \mathsf{Int}(\mathcal{P})\cup Z$. Hence, $P_{k_1,k_2}$ satisfies \ref{1pathlab2} and \ref{1pathlab3}, a contradiction to the maximum of $\mathcal{P}$.

Therefore, there exists a center $u_i$ connected to $2m^{y-z}$ distinct centers $v_j$, say $v_1,\ldots v_{2m^{y-z}}$, which correspond to stars $T_1, \ldots T_{2m^{y-z}}$. Recall that all stars in $\{T_1, \ldots T_{2m^{y-z}}\}$ are vertex-disjoint and the number of vertices in all $P_{i,j}$ $(j\in [2m^{y-z}])$ is at most $2m^{y-z}(m+2)<\frac{dm^z}{10\gamma}\leq \frac{1}{2}e(T_i)$ (as $y<2z+9$ and $\gamma<\frac{d}{m^{10}}$). Hence, every $T_j$ $(j\in[2m^{y-z}])$ has at least $\frac{dm^z}{10\gamma}$ leaves that are not used in $P_{i,j}$ for any $j\in [2m^{y-z}]$. These stars, together with the corresponding paths to $u_i$, form a desired unit in $G-W-X'$. Thus, we can greedily pick vertex-disjoint units as above.
\end{proof}

Applying Claim \ref{lem-b}, we get pairwise vertex-disjoint $(2m^{y-z},\frac{dm^z}{10\gamma},m+2)$-units $F_{1},\ldots,$ $F_{t}$ with $t=100\gamma m^{x-2y+z-3}$, and denoted by $u_i$ the core vertex of $F_i$.
Let $Y=\bigcup_{i\in [t]}V(F_i)$ and $Y'=\bigcup_{i\in [t]}\mathsf{Int}(F_i)$. Since $m^z\leq\gamma<\frac{d}{m^{10}}$, we have
\begin{align}
|Y|&\leq 100\gamma m^{x-2y+z-3}\left(2m^{y-z}(m+2+\frac{dm^z}{10\gamma})\right)\nonumber \\
&=200\gamma m^{x-y-3}(m+2)+20dm^{x-y+z-3}\nonumber \\
&\leq 21dm^{x-y+z-3},\nonumber
\end{align}
and
\begin{equation*}
|Y'|\leq 100\gamma m^{x-2y+z-3}\cdot2m^{y-z}=200\gamma m^{x-y-3}.
\end{equation*}
By Claim \ref{lem-a}, we can greedily find $m^{x-2y+z-3}$ disjoint stars $S_{1},\ldots,S_{m^{x-2y+z-3}}$ which are disjoint from $W\cup Y$, where each $S_{i}$ has exactly $\frac{d}{4}$ leaves and its center vertex, say $v_i$, lies in $V_1$. For simplicity, let $Z=\{v_1,\ldots,v_{m^{x-2y+z-3}},u_1,\ldots,u_{t}\}$.

Let $\mathcal{Q}$ be a maximum collection of internally disjoint paths $Q_{ij}$ satisfying the following rules.
\stepcounter{propcounter}
\begin{enumerate}[label = ({\bfseries \Alph{propcounter}\arabic{enumi}})]
\rm\item\label{wpathlab2} Each path $Q_{ij}$ in $\mathcal{Q}$ is a unique $v_i,u_j$-path of length at most $4m$.
\rm\item\label{wpathlab3} Each $Q_{ij}$ does not contain any vertex in $W\cup Z\cup (Y'\backslash(\mathsf{Int}(F_i)\cup\mathsf{Int}(F_j)))$ as an internal vertex.
\end{enumerate}
\begin{claim}\label{webQ}
There is a center $v_i$ connected to at least $44\gamma$ distinct centers $u_j$ via the paths in $\mathcal{Q}$.
\end{claim}
\begin{proof}[Proof of Claim \ref{webQ}]\renewcommand*{\qedsymbol}{$\blacksquare$}
Suppose to the contrary that each $v_i$ is connected to less than $44\gamma$ centers $u_j$. Then $|\mathcal{Q}|\leq 44\gamma m^{x-2y+z-3}$ and $|V(\mathcal{Q})|\leq 44\gamma m^{x-2y+z-3}\cdot4m=176\gamma m^{x-2y+z-2}$. Let
\begin{equation*}
V:=\left(\bigcup_{i\in [m^{x-2y+z-3}]}(S_i\backslash\{v_i\})\right)\backslash V(\mathcal{Q})
\end{equation*}
and $U$ be the set of exteriors of all units $F_i$ whose centers are not used as endpoints of path $Q_{ij}$ in $\mathcal{Q}$. Then we have
\begin{equation*}
|V|=\frac{dm^{x-2y+z-3}}{4}-176\gamma m^{x-2y+z-2}>\frac{dm^{x-2y+z-3}}{10},
\end{equation*}
and
\begin{equation*}
|U|\geq 2m^{y-z}\cdot\frac{dm^z}{10\gamma}\cdot(100\gamma m^{x-2y+z-3}-44\gamma m^{x-2y+z-3})> \frac{dm^{x-2y+z-3}}{10}.
\end{equation*}
On the other hand,
\begin{align}
&|W|+|\mathsf{Int}(\mathcal{Q})|+|Y'|+|Z| \nonumber \\
&\leq 100dm^{x-2y+z-4}+176\gamma m^{x-2y+z-2}+200\gamma m^{x-y-3}+m^{x-2y+z-3}+100\gamma m^{x-2y+z-3} \nonumber \\
&\leq  \frac{1}{4}\rho\left(\frac{dm^{x-2y+z-3}}{10}\right)\cdot \frac{dm^{x-2y+z-3}}{10}. \nonumber
\end{align}
Hence, applying Lemma \ref{distance} with $V,U,W\cup \mathsf{Int}(\mathcal{Q})\cup Z$ playing the roles of $X_1,X_2,W$, respectively, we obtain vertices $y_{k_1}\in V, y_{k_2}\in U$ and a path of length at most $m$ connecting $y_{k_1}$ and $y_{k_2}$ whilst avoiding vertices in $W\cup \mathsf{Int}(\mathcal{Q})\cup Z$. Denote by $S_{k_1}$ the star which contains $y_{k_1}$ as a leave and $F_{k_2}$ the unit such that $y_{k_2}\in \mathsf{Ext}(F_{k_2})$. This yields a $v_{k_1},u_{k_2}$-path, denoted as $Q_{k_1,k_2}$, which is internally disjoint from $W\cup \mathsf{Int}(\mathcal{Q})\cup Z\cup (Y'\backslash\mathsf{Int}(F_{k_2}))$. Hence, $Q_{k_1,k_2}$ satisfies \ref{wpathlab2} and \ref{wpathlab3}, a contradiction to the maximum of $\mathcal{Q}$.
\end{proof}

By Claim \ref{webQ}, there is a center $v_i$ connected to $44\gamma$ distinct centers $u_j$, say $u_1,\ldots u_{44\gamma}$, which is corresponding to units $F_1, \ldots F_{44\gamma}$. Let $\mathcal{Q}'$ be the family of all the paths $Q_{ij}$, where $j\in[44\gamma]$. A pendent star $S$ in a unit $F_i$ $(i\in[44\gamma])$ is \emph{overused} if at least $\frac{dm^z}{20\gamma}$ leaves of $S$ are used in $V(\mathcal{Q'})$, and a unit $F_i$ is \emph{bad} if at least $m^{y-z}$ stars are overused. Note that the number of bad units is at most
\begin{equation*}
\frac{|V(\mathcal{Q}')|}{m^{y-z}\cdot\frac{dm^z}{20\gamma}}\leq\frac{176\gamma m}{m^{y-z}\cdot\frac{dm^z}{20\gamma}}<22\gamma.
\end{equation*}
Hence, there are at least $22\gamma$ units among $F_1,\ldots F_{44\gamma}$, where each pendent star has at least $\frac{dm^z}{20\gamma}$ leaves that are not used in $V(\mathcal{Q}')$. For each of these units, we take a sub-unit by including the branches each attached with a pendant star that is not overused to form a desired $(22\gamma, m^{y-z}, \frac{dm^z}{20\gamma},4m)$-web in $G-W$.
\end{proof}

\section{Proof of Lemmas \ref{link} and \ref{extadju}}\label{appen-adj}

\begin{proof}[Proof of Lemma \ref{extadju}]
Given $\varepsilon_1,\varepsilon_2,s,x$ such that $s\geq 8x$, we choose $\frac{1}{K}\ll c\ll \varepsilon_1,\varepsilon_2$, and let $G$ be an $n$-vertex $(dm^x,\frac{d}{2})$-dense $(\varepsilon_1,\varepsilon_2d)$-expander with $\delta(G)\geq d$. For any $W\subseteq V(G)$ with $|W|\leq Dm^{-\frac{3}{4}}$. Let $Z_1,Z_2\subseteq V(G)\backslash W$ be two vertex-disjoint sets and each of size at least $D$. For each $j\in[2]$, let $I_j\subseteq V(G)\backslash(W\cup Z_1\cup Z_2)$ be an $(D,m)$-expansion centered at $v_j$. Notice that $|W|\leq \frac{D}{m^{\frac{3}{4}}}\leq \frac{\rho(2D)2D}{4}$, $|Z_1\cup Z_2|\geq 2D$, and $|I_1\cup I_2|=2D$. Thus, there is a path $P'_1$ of length at most $m$ avoiding $W$ from $Z_1\cup Z_2$ to $I_1\cup I_2$ by Lemma \ref{distance}, say $P'_1$ is a $z'_1,v'_1$-path, where $z'_1\in Z_1$, $v'_1\in I_1$. Since $I_1$ is a $(D,m)$-expansion centered at $v_1$, $P'_1$ can be extended to a $z'_1,v_1$-path $P$ of length at most $2m$. Now denote by $W':=W\cup V(P)$, and we have $|W'|\leq 2Dm^{-\frac{3}{4}}$.
\begin{claim}\label{le10m}
There is a $u,v_2$-path in $G-W'$ for some $u\in Z_2$ of length between $\ell$ and $\ell+16m$ for any $\ell\leq dm^{y-2}$.
\end{claim}
\begin{proof}[Proof of Claim \ref{le10m}]\renewcommand*{\qedsymbol}{$\blacksquare$}
Let $(P^{*},v^{*},F_1)$ be a triple such that $\ell(P^{*})$ is maximised and satisfying the following rules.
\stepcounter{propcounter}
\begin{enumerate}[label = ({\bfseries \Alph{propcounter}\arabic{enumi}})]
\rm\item\label{triplelab1} $F_1$ is a $(D,3m)$-expansion centered at $v^{*}$ in $G-W$.
\rm\item\label{triplelab2} $P^{*}$ is a $v_2,v^{*}$-path in $G-W$ and $V(F_1)\cap V(P^{*})=\{v^{*}\}$.
\rm\item\label{triplelab3} $\ell(P^{*})\leq \ell+12m$.
\end{enumerate}
Taking $F_1:=I_2$, $v^{*}:=v_2$, $P=G[\{v^{*}\}]$ trivially satisfies all the conditions, thus such a triple exists. We first claim that $\ell(P^{*})\geq \ell$.
Otherwise we denote $W_1:=W\cup V(P^{*})\cup V(F_1)$, and then $|W_1|\leq 2Dm^{-\frac{3}{4}}+\ell+D\leq 2D$. By Lemma \ref{middlewebs}, $G-W_1$ contains a $(22\gamma,m^{y-z},\frac{dm^z}{20\gamma},4m)$-web $F'$ with core vertex $v$. However, $|W\cup V(P^{*})|\leq 2Dm^{-\frac{3}{4}}+\ell\leq 3Dm^{-\frac{3}{4}}\leq \frac{1}{4}\rho(n)D\leq \frac{1}{4}\rho(D)D$, and $|F'|\geq 22dm^y\geq D$, $|F_1|\geq D$ also hold. Thus, by Lemma \ref{distance}, there is a $u'_1,u'_2$-path $Q'$ of length at most $m-1$, where $u'_1\in V(F_1)$ and $u'_2\in V(F')$, and so $Q'$ can be extended to a $v,v^{*}$-path $Q$ of length at most $3m+m-1+8m+1=12m$. By the property of the web $F'$, we know that there exists a $F_2\subseteq (F'\backslash V(Q))\cup\{v\}$ which is a $(D,9m)$-expansion centered at $v$. Now let $P'=P^{*}\cup Q$ which is a $v,v_2$-path with $\ell(P^{*})+1\leq\ell(P')\leq \ell(P^{*})+12m<\ell +12m$. Thus, we find a triple $(P',v_2,F_2)$ satisfying three conditions \ref{triplelab1}-\ref{triplelab3} with $\ell(P')>\ell(P^{*})$, a contradiction to the maximality of $\ell(P^{*})$. Hence, $\ell(P^{*})\geq\ell$, as claimed.

Note that $|W\cup V(P^{*})|\leq 2Dm^{-\frac{3}{4}}+\ell+12m\leq 3Dm^{-\frac{3}{4}}\leq \frac{1}{4}\rho(n)D\leq \frac{1}{4}\rho(D)D$, $|F_1|\geq D$ and $|Z_2|\geq D$. By Lemma \ref{distance}, there is a $r_1,r_2$-path $Q_1$ of length at most $m$ avoiding $W\cup V(P^{*})$, where $r_1\in Z_2$ and $r_2\in F_1$. Let $Q_2$ be a $v^{*},r_2$-path. Thus, $Q_1\cup Q_2\cup P$ is a $v_2,r_1$-path in $G-W$ satisfying $\ell\leq\ell(Q_1\cup Q_2\cup P^{*})\leq\ell(P^{*})+3m+m\leq \ell+16m$. Finally, the claim holds by taking $u:=r_1$.
\end{proof}
By Claim \ref{le10m}, we can find a $u,v_2$-path $Q$ satisfying $\ell\leq\ell(Q)\leq\ell+16m$ while avoiding $W'$, where $u\in Z_2$. Therefore, $\ell\leq\ell(P)+\ell(Q)\leq \ell+18m$, and such $P,Q$ are as desired.
\end{proof}

We now turn to Lemma~\ref{link}. We need the following simple fact about expansions.

\begin{proposition}[\cite{LTWY}]\label{expope}
Let $D,D',m\in \mathbb{N}$ and $1\leq D'\leq D$. Then any graph $F$ which is a $(D,m)$-expansion centered at $v$ contains a subgraph which is a $(D',m)$-expansion centered at $v$.
\end{proposition}

The following definition is essential to find a large adjuster.
\begin{definition}[\cite{LTWY}]
Given $r_1,r_2,r_3,r_4\in \mathbb{N}$, an \emph{$(r_1,r_2,r_3,r_4)$-octopus} $\mathcal{O}=(A,R,\mathcal{D},\mathcal{P})$ is a graph consisting of a \emph{core} $(r_1,r_2,1)$-adjuster $A$, one of the ends of $A$, called $R$ and
\begin{enumerate}
\item[$\bullet$] a family $\mathcal{D}$ of $r_3$ vertex-disjoint $(r_1,r_2,1)$-adjusters, which are disjoint from $A$, and
\item[$\bullet$] a minimal family $\mathcal{P}$ of internally vertex-disjoint paths of length at most $r_4$, such that each adjuster in $\mathcal{D}$ has at least one end which is connected to $R$ by a subpath from a path in $\mathcal{P}$, and all of the paths are disjoint from all center sets of the adjusters in $\mathcal{D}\cup A$. Obviously, $|\mathcal{P}|\leq|\mathcal{D}|$.
\end{enumerate}
\end{definition}

The following lemma is the $r=1$ case of Lemma~\ref{link}, and we postpone its proof to the end.
\begin{lemma}\label{simpleadj}
Suppose $\frac{1}{n}, \frac{1}{d}\ll\frac{1}{K}\ll \varepsilon_1,\varepsilon_2<1$ and $s,x,y\in \mathbb{N}$ satisfy $s\geq 1600$, $s\geq 8x\geq 16y$ and $\log^sn\leq d\leq \frac{n}{K}$. Let $m=\log^4\frac{n}{d}$ and $D=10^{-7}dm^{y}$. If $G$ is an $n$-vertex $(dm^x,\frac{d}{2})$-dense $(\varepsilon_1,\varepsilon_2d)$-expander with $\delta(G)\geq d$, and $W'\subseteq V(G)$ satisfies $|W'|\leq 10D$, then $G-W'$ contains a $(D,\frac{m}{4},1)$-adjuster.
\end{lemma}

\begin{proof}[Proof of Lemma \ref{link}]
Given $\varepsilon_1,\varepsilon_2,s,x,y$ such that $s\geq 8x>8y$ and $s\geq 1600$, we choose $\frac{1}{K}\ll \varepsilon_1,\varepsilon_2$, and let $G$ be an $n$-vertex $(dm^x,\frac{d}{2})$-dense $(\varepsilon_1,\varepsilon_2d)$-expander with $\delta(G)\geq d$. Take a set $W\subseteq V(G)$ with $|W|\leq Dm^{-\frac{3}{4}}$. We prove the lemma holds by induction on $r$.

Suppose that for some $1\leq r<\frac{dm^{y-2}}{10}$, $G-W$ contains a $(D,m,r)$-adjuster, denoted by $\mathcal{A}_1:=(v_1,F_1,v_2,F_2,A_1)$. Let $W_1=W\cup V(F_1)\cup V(F_2)\cup A_1$. Then $|W_1|\leq 4D$. By Lemma \ref{simpleadj}, there is a $(D,\frac{m}{4},1)$-adjuster $\mathcal{A}_2:=(v_3,F_3,v_4,F_4,A_2)$ in $G-W_1$. As $|F_1\cup F_2|=|F_3\cup F_4|=2D$ and $|W\cup A_1\cup A_2|\leq \frac{D}{m^{\frac{3}{4}}}+10mr+10m\leq 2dm^{y-1}\leq \frac{\rho(2D)2D}{4}$, there is a path $P'$ of length at most $m$ from $F_1\cup F_2$ to $F_3\cup F_4$ avoiding $W\cup A_1\cup A_2$, say that $P'$ is a $v'_1,v'_3$-path with $v'_1\in F_1$, $v'_3\in F_3$. Using that $F_1$ and $F_3$ are $(D,m)$-expansion centered at $v_1$ and $v_3$, respectively, $P'$ can be extended to be a $v_1,v_3$-path $P$ of length at most $3m$. We claim that $(v_2,F_2,v_4,F_4,A_1\cup A_2\cup P)$ is a $(D,m,r+1)$-adjuster. Indeed, we easily have that \ref{adjla1} and \ref{adjla2} hold, and $|A_1\cup A_2\cup P|\leq 10mr+10\cdot\frac{m}{4}+3m\leq 10m(r+1)$, so that \ref{adjla3} holds. Finally, let $\ell=\ell(\mathcal{A}_1)+\ell(\mathcal{A}_2)+\ell(P)$. If $i\in\{0,1,\ldots,r+1\}$, then there is some $i_1\in\{0,1,\ldots,r\}$ and $i_2\in\{0,1\}$ with $i=i_1+i_2$. Let $P_1$ be a $v_1,v_2$-path of length $\ell(\mathcal{A}_1)+2i_1$ in $G[A_1\cup\{v_1,v_2\}]$ and $P_2$ be a $v_3,v_4$-path of length $\ell(\mathcal{A}_2)+2i_2$ in $G[A_2\cup\{v_3,v_4\}]$. Thus, $P_1\cup P\cup P_2$ is a $v_2,v_4$-path of length $\ell+2i$ in $G[A_1\cup A_2\cup V(P)]$, and so $\ell$ satisfies \ref{adjla4}.
\end{proof}

\begin{proof}[Proof of Lemma \ref{simpleadj}]
Given $\varepsilon_1,\varepsilon_2,s,x,y$ such that $s\geq 8x\geq 16y$ and $s\geq 1600$, we choose $\frac{1}{K}\ll \varepsilon_1,\varepsilon_2$, and fix $G$ to be an $n$-vertex $(dm^x,\frac{d}{2})$-dense $(\varepsilon_1,\varepsilon_2d)$-expander with $\delta(G)\geq d$. Take a set $W'\subseteq V(G)$ with $|W'|\leq 10D$. First, the following claim allows us to find many adjusters in $G-W'$.
\begin{claim}\label{appendix}
There are $m^{x}$ pairwise disjoint $(\frac{d}{800},\frac{m}{400},1)$-adjusters in $G-W'$.
\end{claim}
\begin{proof}[Proof of Claim \ref{appendix}]\renewcommand*{\qedsymbol}{$\blacksquare$}
Suppose that there are less than $m^{x}$ vertex-disjoint $(\frac{d}{800},\frac{m}{400},1)$-adjusters as above, and denote by $W_0$ the vertices of all such adjusters. Let $W=W'\cup W_0$, and then $|W|\leq dm^{y}+ m^{x}(2\cdot\frac{d}{800}+10\cdot\frac{m}{400})\leq dm^{\frac{x}{2}}+\frac{dm^{x}}{20}\leq dm^x$. By the assumption that $G$ is $(dm^x,\frac{d}{2})$-dense, we have $d(G-W)\geq \frac{d}{2}$, and by Corollary \ref{expander}, there exists a bipartite $(\varepsilon_1,\varepsilon_2d)$-expander $G_1\subseteq G-W$ with $\delta(G_1)\geq \frac{d}{16}$. Thus, there exists a shortest cycle $C$ in $G_1$ of length at most $\frac{m}{40}$, and denote by $2r$ the length of $C$. Now we arbitrarily choose two vertices $v_1,v_2\in V(C)$ of distance $r-1$ apart on $C$, together with $\frac{d}{800}$ distinct vertices in $N_{G_1-C}(v_1)$ and $N_{G_1-C}(v_2)$, respectively, and then we get a $(\frac{d}{800},\frac{m}{400},1)$-adjuster as desired.
\end{proof}

An adjuster is \emph{touched} by a path if they intersect on at least one vertex, and \emph{untouched} otherwise.
\begin{claim}\label{linkmany}
Let $G,m,d$ be as above. For integers $t,y$ with $t\geq y+1$, let $X\subseteq V(G)$ be an arbitrary set with $|X|\leq \frac{dm^{t-1}}{2}$, $Y\subseteq V(G)-X$ with $|Y|\geq \frac{dm^t}{800}$, and $\mathcal{U}$ be a family of $(\frac{d}{800},\frac{m}{400},1)$-adjusters with $|\mathcal{U}|\geq 210m^{2t}$ in $G-(X\cup Y)$. Let $\mathcal{P}_{Y}$ be a maximum collection of internally vertex-disjoint paths of length at most $\frac{m}{8}$ in $G-X$, where each path connects $Y$ to one end from distinct adjusters in $\mathcal{U}$. Then $Y$ can be connected to $1600m^{t+y}$ ends from distinct adjusters in $\mathcal{U}$ via a subpath of a path $P\in\mathcal{P}_{Y}$.
\end{claim}
\begin{proof}[Proof of Claim \ref{linkmany}]\renewcommand*{\qedsymbol}{$\blacksquare$}
Suppose to the contrary that $Y$ is connected to less than $1600m^{t+y}$ ends from distinct adjusters in $\mathcal{U}$ via a subpath of a path $P\in\mathcal{P}_{Y}$. Let $Q$ be the set of all internal vertices of those paths, then $|Q|\leq 1600m^{t+y}\cdot\frac{m}{8}=200m^{t+y+1}$, and $|X\cup Q|\leq dm^{t-1}\leq \frac{1}{4}\cdot\rho(\frac{dm^t}{800})\frac{dm^t}{800}$, and so there are at least  $210m^{2t}-200m^{t+y+1}\geq m^{t}$ adjusters in $\mathcal{U}$ untouched by the paths in $\mathcal{P}_{Y}$. Choose arbitrarily $m^t$ such adjusters, and let $Z$ be the vertex set of the union of their ends. We get $|Z|=m^t\cdot 2\cdot\frac{d}{800}=\frac{dm^t}{400}\geq \frac{dm^t}{800}$. Since $|Y|\geq \frac{dm^t}{800}$, Lemma \ref{distance} implies that there is a path of length at most $\frac{m}{8}$ between $Y$ and $Z$  avoiding $X\cup Q$, a contradiction to the maximality of $\mathcal{P}_Y$.
\end{proof}

By Claim \ref{appendix}, we have found many adjusters, and we aim to construct many octopuses via those $(\frac{d}{800},\frac{m}{400},1)$-adjusters we found above. Let $Z$ be the union of the center sets and core vertices of all those adjusters.
\begin{claim}\label{manyoctopus}
For integers $x,y,z$ with $2y< y+z<\frac{x}{2}$, there are $m^{z}$ $(\frac{d}{800},\frac{m}{400},800m^{y},\frac{m}{8})$-octopus $\mathcal{O}_j=(A_j,R_j,\mathcal{D}_j,\mathcal{P}_j)$ $(1\leq j\leq m^{z})$ in $G-W$ such that the following rules hold.
\stepcounter{propcounter}
\begin{enumerate}[label = ({\bfseries \Alph{propcounter}\arabic{enumi}})]
\rm\item\label{ocolab1} $A_j$ are pairwise disjoint adjusters, $1\leq j\leq m^{z}$.
\rm\item\label{ocolab2} $A_i\notin \mathcal{D}_j$, $1\leq i,j\leq m^{z}$.
\rm\item\label{ocolab3} $\mathcal{D}_j$ contains every adjusters other than $A_j$ which intersects with a path in $\mathcal{P}_j$, $1\leq j\leq m^{z}$.
\rm\item\label{ocolab4} Paths in $\mathcal{P}_i$ are vertex disjoint from $Z$ and $A_j$, $1\leq i\neq j\leq m^{z}$.
\rm\item\label{ocolab5} Every two paths from distinct $\mathcal{P}_i,\mathcal{P}_j$ are mutually vertex disjoint, $1\leq i< j\leq m^{z}$.
\end{enumerate}
\end{claim}

\begin{proof}[Proof of Claim \ref{manyoctopus}]\renewcommand*{\qedsymbol}{$\blacksquare$}
We aim to construct the desired octopuses iteratively. Suppose that we have constructed $t$ $(<m^{z})$ octopuses. Let $W_1=W'\cup Z$, and $|W_1|\leq 10D+m^{x}(\frac{m}{40}+2)\leq 12D$. Let $U$ be the union of the vertex sets of the ends in the core adjusters of octopuses we have found, and $|U|\leq t\cdot 2\cdot\frac{d}{800}< \frac{dm^{z}}{400}$. For simplicity, an adjuster is \emph{used} if it is an element of an octopus found so far, and \emph{unused} otherwise. Until now, we know that there are at most $m^{z}(800m^{y}+1)\leq 810m^{z+y}$ used adjusters, and thus at least $m^{\frac{x}{3}}$ (as $z+y\leq \frac{x}{2}$) unused adjusters.

Arbitrarily choose $m^{a}$ unused adjusters for some $a\geq y+1$, denoted by $\mathcal{B}$, and let $X$ be the union of the vertex sets of the ends of all adjuster in $\mathcal{B}$. Then $|X|=m^{a}\cdot2\cdot\frac{d}{800}=\frac{dm^{a}}{400}$. Note that there are at least $210m^{\frac{x}{3}-a}$ unused adjusters remained apart from $\mathcal{B}$, and denoted them by $\mathcal{U}$. Let $Q=\bigcup_{j=1}^tV(\mathcal{P}_j)$, and then $|Q|\leq m^{z}\cdot 800m^{y}\cdot\frac{m}{8}\leq m^{z+y+1}$. Thus, $|W_1\cup U\cup Q|\leq 12D+\frac{dm^{z}}{400}+m^{z+y+1}\leq \frac{dm^{z}}{2}$ as $y<z$. Applying Claim \ref{linkmany} with $(Y,\mathcal{U},t,X)=(X,\mathcal{U},a,W_1\cup U\cup P)$, respectively, we get that $X$ can be connected to $1600m^{a+y}$ ends from different adjusters in $\mathcal{U}$ via some internally vertex-disjoint paths of length at most $\frac{m}{8}$ in $G-W_1-U-Q$. Thus, there exists an adjuster in $\mathcal{B}$, say $A_{t+1}$, such that $A_{t+1}$ has an end $R_{t+1}$ connected to a family $\mathcal{D}_{t+1}$ of at least $800m^{y}$ adjusters via a subfamily $\mathcal{P}_{t+1}$ of internally vertex-disjoint paths. By the construction, \ref{ocolab1}-\ref{ocolab5} obviously hold. That is, $A_{t+1},R_{t+1},\mathcal{D}_{t+1},\mathcal{P}_{t+1}$ form a $(\frac{d}{800},\frac{m}{400},800m^{y},\frac{m}{8})$-octopus.
\end{proof}

Now we have $m^{z}$ octopuses $\mathcal{O}_j=(A_j,R_j,\mathcal{D}_j,\mathcal{P}_j)$. Let $L_{j}\neq R_j$ be another end of $A_{j}$, and $X'=\bigcup_{i\in[m^{z}]}V(L_i)$. Then $|X'|=\frac{dm^{z}}{800}$. As we have found $m^{x}$ adjusters and at most $m^{z}\cdot(800m^y+1)$ used adjusters, there are at least $210m^{\frac{x}{3}}$ unused adjusters $\mathcal{U}'$. Let $Q'=\bigcup_{j=1}^{m^{z}}V(\mathcal{P}_j)$, and $|Q'|\leq m^{\frac{x}{3}}$. Note that for each adjuster $A\in \mathcal{D}_j$, there is a path $P_j\in \mathcal{P}_j$, and $|W_1\cup Q'|\leq 12D+m^{\frac{x}{3}}\leq \frac{dm^{y}}{2}$. Applying Claim \ref{linkmany} with $(X,\mathcal{U},t,Y)=(X',\mathcal{U}',z,W_1\cup Q')$, respectively, we know that $X'$ can be connected to $800m^{z+y}$ ends from distinct adjusters in $\mathcal{U}'$ via internally vertex-disjoint paths of length at most $\frac{m}{8}$ in $G-(W_1\cup Q')$. Hence, there exists an adjuster $A_k$ such that $L_k$ is connected to a family $\mathcal{U}'_k$ of at least $800m^{y}$ adjusters via a subfamily of internally vertex-disjoint paths, denoted by $\mathcal{P}'_k$. Thus, $A_k,L_k,\mathcal{U}'_k,\mathcal{P}'_k$ form an $(\frac{d}{800},\frac{m}{400},800m^{y},\frac{m}{8})$-octopus. Note that $A_k,R_k,\mathcal{U}_k,\mathcal{P}_k$ also form a $(\frac{d}{800},\frac{m}{400},800m^{y},\frac{m}{8})$-octopus.

Let $F'_1=G[V(L_k)\cup V(\mathcal{P}'_k)\cup V(\mathcal{U}'_k)]$ and $F'_2$ be the component of $G[V(R_k)\cup V(\mathcal{P}_k)\cup V(\mathcal{U}_k)]-V(\mathcal{P}'_k)$ containing $v_2$. Indeed, $V(\mathcal{P}_k)\cap V(\mathcal{P}'_k)=\varnothing$. As $V(\mathcal{P}'_k)$ is disjoint from $Z$ and $Q'$, $F'_2$ has size at least
\begin{equation*}
|V(\mathcal{U}_k)|-|V(\mathcal{P}'_k)|\geq 800m^{y}\cdot2\cdot\frac{d}{800}-\frac{m}{8}\cdot800m^{y}\geq dm^{y},
\end{equation*}
and the distance between $v_2$ and each $v\in V(F'_2)$ is at most $\frac{m}{400}+\frac{m}{8}+\frac{m}{400}+\frac{m}{32}+\frac{m}{400}\leq \frac{m}{4}$. Then by Proposition \ref{expope}, there exists a subgraph $F_2$ of $F'_2$, which is a $(dm^{y},\frac{m}{4})$-expansion centered at $v_2$. Similarly, we can find $F_1\subseteq F'_1$, which is a $(dm^{y},\frac{m}{4})$-expansion centered at $v_1$. For $A_k$, denote by $C_k$ the center vertex set of $A_k$. Recall that $C_k\cup \{v_1,v_2\}$ is an even cycle of length $2r\leq \frac{m}{16}$, and the distance between $v_1$ and $v_2$ on $C_k\cup \{v_1,v_2\}$ is $r-1$. Hence, $(v_1,F_1,v_2,F_2,C_k)$ is a $(dm^{y},\frac{m}{4},1)$-adjuster. By Proposition \ref{expope}, there exists a $(D,\frac{m}{4},1)$-adjuster in $G-W$.
\end{proof}

\section{Proof of Lemma \ref{inte}}\label{appen-lem:3.6}
\begin{proof}[Proof of Lemma \ref{inte}]
We take $x=50$ and choose $\frac{1}{h}, \frac{1}{f}, \frac{1}{r} \ll \frac{1}{K},\alpha, c\ll\frac{1}{\kappa},\varepsilon,\varepsilon_1,\varepsilon_2<1$ and $n,d\in\mathbb{N}$ satisfy $\log^{s}n<d<\frac{n}{K}$. Let $H$ be an $h$-vertex graph with $d(H)\geq K$ and the biseparability constraints as in  \ref{spenu1}-\ref{spenu2} and $G$ be an $n$-vertex bipartite $(\varepsilon_1,\varepsilon_2 d)$-expander with $\delta(G)=d\geq \varepsilon e(H)$. Write $m=\log^4\frac{n}{d}$. Then we have $h\leq \frac{2d}{\varepsilon d(H)}\leq \frac{\varepsilon_1\varepsilon_2d}{20}$ by the choice of $K$. We further assume that $G$ is locally-$(dm^{50},\frac{d}{2})$-dense as otherwise Lemma \ref{dense34} implies $TH\subseteq G$ as desired.

Let $L_G:=\{v\in V(G):d_G(v)\geq 2dm^{12}\}$, where $m=\log^4\frac{n}{d}$. We divide the proof into two cases depending on whether there are many large degree vertices.

\noindent
\textbf{Case 1: $|L_G|\geq h$.}
Let $V(H)=\{x_1,\ldots,x_h\}$ and $E(H)=\{e_1,\ldots,e_q\}$ with $q=e(H)$. Hence, we can take a set $Z=\{u_1,\ldots,u_h\}$ of $h$ distinct vertices in $L_G$. Let $\tau: V(H)\rightarrow Z$ be an arbitrary injection. Note that for each $i\in[h]$, the set $N(u_i)$ has size at least $2dm^{12}$. Next we shall construct a $TH$ by greedily finding a collection of internally vertex-disjoint paths. Assume that we have pairwise internally disjoint paths of length at most $2m$, say $P(e_1),\ldots, P(e_t)$, such that $t\leq q$ and each $P(e_j)$ connects the two vertices in $\tau(e_j)$ whilst $P(e_j)$ is internally disjoint from $Z$.

Let $W=\bigcup_{j\in[t]}\mathsf{Int}(P(e_j))$ be the union of the interior vertices of the paths. Then $|W|+|Z|< 2qm+h\leq \frac{1}{4}\rho(2dm^{12})2dm^{12}$. We can apply Lemma \ref{distance} to get a path $P$ avoiding $W\cup Z$ of length at most $m$, and extend $P$ to obtain a path $P_{t+1}$ of length at most $m+2\leq 2m$. Repeating this for $t=0,1,\ldots,q$ in order, we obtain $\bigcup_{j\in[q]}P_j$, which is an $H$-subdivision in $G$.

\noindent
\textbf{Case 2: $|L_G|\leq h$.}
We choose $\frac{1}{K}\ll c_0\ll \varepsilon$ so that $h\le c_0d$. Denote by $\mathbf{L}=\{v\in V(H)|d(v)> \frac{d}{m^{10}}\}$, $\mathbf{M}=\{v\in V(H) \mid m^2\leq d(v)\leq \frac{d}{m^{10}}\}$ and $\mathbf{S}=\{v\in V(H) \mid d(v)< m^2\}$. Note that
\begin{equation*}
2e(H)=\sum_{v\in V(H)}d_H(v)\geq |\mathbf{L}|\cdot \tfrac{d}{m^{10}},
\end{equation*}
 and thus $|\mathbf{L}|\leq \frac{2m^{10}}{\varepsilon}$.

As usual, we shall find units and webs for each vertex in $H$. First, Applying Lemma \ref{middlewebs} with $x=50, y=12,z=4$, we can greedily find a family of internally vertex-disjoint webs $\{Z_v\}_{v\in\mathbf{M}}$, where $Z_v$ is a $(22d_H(v),m^{8},\frac{dm^4}{20d_H(v)},4m)$-web and $2|\mathbf{S}|$ internally vertex-disjoint $(22m^4,m^{8},\frac{d}{20},4m)$-webs $Z_1, \ldots, Z_{2|\mathbf{S}|}$. Indeed, this can be done by repeatedly applying Lemma \ref{middlewebs} to $G$ with $W_0$ being the set of internal vertices of objects found so far and by the fact that
\[|W_0| \leq \sum_{v\in\mathbf{M}}90 d_H(v)m^9+44|\mathbf{S}|m^{13}\le 180 e(H)m^{9}+ 44h m^{13}< 100dm^{13}.\] Next, we need the following claim.
\begin{claim}\label{subexpan1}
$G_1:=G-L_G$ is an $(\frac{\varepsilon_1}{2},\varepsilon_2d)$-expander satisfying $\delta(G_1)\geq \frac{d}{2}$ and $|G_1|\geq \frac{n}{2}$.
\end{claim}
\begin{proof}\renewcommand*{\qedsymbol}{$\blacksquare$}
Recall that $|L_G|\leq h<\varepsilon_2 d$, then $|G_1|\geq n-|L_G|\geq \frac{n}{2}$, and $\delta(G_1)\geq \delta(G)-|L_G|\geq \frac{d}{2}$. It remains to show that $G_1$ is an $(\frac{\varepsilon_1}{2},\varepsilon_2d)$-expander. Since $G$ is an $(\varepsilon_1,\varepsilon_2d)$-expander and $\rho(x)x$ is increasing when $x\geq \frac{\varepsilon_2d}{2}$, for any set $X$ in $G_1$ of size $x\geq \frac{\varepsilon_2d}{2}$ with $x\leq \frac{|G_1|}{2}\leq \frac{|G|}{2}$, we have
$$
|N_G(X)|\geq x\cdot\rho(x,\varepsilon_1,\varepsilon_2d)\geq \tfrac{\varepsilon_2d}{2}\cdot\rho\left(\tfrac{\varepsilon_2d}{2},\varepsilon_1,\varepsilon_2d\right)
=\tfrac{\varepsilon_2d}{2}\cdot \tfrac{\varepsilon_1}{\log^2(\frac{15}{2})} \geq \tfrac{\varepsilon_1\varepsilon_2d}{10}\geq 2h.
$$
Hence, $|N_{G_1}(X)|\geq |N_G(X)|-|L_G|\geq \frac{1}{2}|N_G(X)|\geq \frac{1}{2}x\cdot\rho(x,\varepsilon_1,\varepsilon_2d)=x\cdot\rho\left(x,\frac{\varepsilon_1}{2},\varepsilon_2d\right)$ as desired.
\end{proof}
As $\Delta(G_1)\le 2dm^{12}$, applying Lemma \ref{bunit} on $G_1$ with $x=14, y=13, z=14$, we can greedily pick a family $\{Z_v\}_{v\in \mathbf{L}}$ of units such that $Z_v$ is a $(c_0d,m^{13},2m)$-unit and the interiors of all units and webs obtained as above are disjoint. This is possible because in the process, the union of $L_G$ and the interiors of all possible units and webs have size at most $dm^{14}$.

Denote by $z_v$ the core vertex of $Z_v$ for each $v\in\mathbf{M}\cup \mathbf{L}$ and $z_i$ the core vertex of the web $Z_i$ for each $i\in[2|\mathbf{S}|]$. In addition,
$$|\mathsf{Ext}(Z_v)|=\begin{cases}
c_0dm^{13}, \quad \text{if } v\in\mathbf{L},\\
\frac{11dm^{12}}{10}, \quad \text{if } v\in\mathbf{M},
\end{cases} \text{and} \quad |\mathsf{Ext}(Z_i)|=\tfrac{11dm^{12}}{10}, \quad \text{if } i\in[2|\mathbf{S}|].$$

Let $H_1$ be the subgraph of $H$ with $V(H_1)=V(H)$ and
\begin{equation*}
E(H_1)=E(H[\mathbf{S}])\cup E(H[\mathbf{S},\mathbf{L}\cup\mathbf{M}]),
\end{equation*}
and write $H_2=H\backslash E(H_1)$. We shall find a mapping $f:V(H)\rightarrow V(G)$ and a family of pairwise internally disjoint paths respecting the adjacencies of $H$ in the following two rounds, where we may abuse the notation $f$ as the up-to-date embedding. To begin, we embed every $v\in \mathbf{L}\cup\mathbf{M}$ by taking $f(v)=z_v$.

\medskip
\noindent \textbf{First round: Finding the desired paths (in $G$) for the adjacencies in $H_1$.}\medskip

\noindent
Let $W=(\bigcup_{v\in \mathbf{L}\cup\mathbf{M}} \mathsf{Int}(Z_v))\cup(\bigcup_{i\in[2|\mathbf{S}|]}\mathsf{Ctr}(Z_i))$. Then as $|\mathbf{L}|\le \frac{m^{10}}{\varepsilon}$ and $c_0\ll \varepsilon$, we have
\[
|W|\leq |\mathbf{L}|\cdot 2c_0dm+\sum_{v\in \mathbf{M}}88d_H(v)m^9+176|\mathbf{S}|m^5\leq 30dm^{11}.
\]
\noindent
For a given vertex set $Y$ and $i\in[2|\mathbf{S}|]$, we say a web $Z_i$ is $Y$-\emph{good} if $|\mathsf{Int}(Z_i)\cap Y|\leq 11m^{12}$. To extend $f$ to $V(H)$ whilst finding the desired paths for the adjacencies in $H_1$, we define $(X,I,I',\mathcal{Q},f)$ to be a good path system if the followings hold.
\stepcounter{propcounter}
\begin{enumerate}[label = ({\bfseries \Alph{propcounter}\arabic{enumi}})]
\rm\item\label{3gopathlab1} $X\subseteq \mathbf{S}$ and $f$ injectively maps the vertex set $X$ to the index set $I\subseteq [2|\mathbf{S}|]$.
\rm\item\label{3gopathlab2} $\mathcal{Q}$ is a collection of internally vertex-disjoint paths  $Q_{x,y}$ of length at most $13m$  for all edges $xy\in E(H_1)$ touching $X$, such that $Q_{x,y}$ is a $z_{f(x)},z_{f(y)}$-path disjoint from $W\setminus ( \mathsf{Int}(Z_{f(x)})\cup\mathsf{Int}(Z_{f(y)}))$.
\rm\item\label{3gopathlab3} In particular, $Q_{x,y}$ begins (or ends) with a subpath $P_{x}(y)$ (resp. $P_{y}(x)$) within the web $Z_{f(x)}$ (resp. $Z_{f(y)}$) connecting the core $z_{f(x)}$ (resp. $z_{f(y)}$) to $\mathsf{Ext}(Z_{f(x)})$. Moreover, we write $Q_{x,y}'$ for the middle segment of $Q_{x,y}$, i.e. $Q_{x,y}'=Q_{x,y}\setminus (P_{x}(y)\cup P_{y}(x))$ and let $\mathcal{Q}'$ be the family of these paths $Q_{x,y}'$.
\rm\item\label{3gopathlab4} $I'=\{i\in [2|\mathbf{S}|]:Z_{i} \ \text{is} \ \text{not} \ V(\mathcal{Q}')\text{-good}\}$ and $I'\cap I=\varnothing$.
\end{enumerate}

Now we shall build a good path system with $X=\mathbf{S}$, and we proceed with our construction as follows.

\noindent
\textbf{Step $0$}. Fix an arbitrary ordering $\sigma$ on $\mathbf{S}$, say the first vertex is $x_1$. Let $X_1=\{x_1\}$, $f(x_1)=1$, $I_1=\{1\}$, $I'_1=\varnothing$ and $\mathcal{Q}_1=\varnothing$. Then $(X_1,I_1,I'_1,\mathcal{Q}_1,f|_{X_{1}})$ is a good path system. Proceed to Step $1$.

\noindent
\textbf{Step $i$}. Stop if either $X_i=\mathbf{S}$ or $I_i\cup I'_i=[2|\mathbf{S}|]$, otherwise we continue:
\stepcounter{propcounter}
\begin{enumerate}[label = ({\bfseries \Alph{propcounter}\arabic{enumi}})]
\rm\item\label{dag1} Let $x$ be the first vertex in $\sigma$ on $\mathbf{S}\backslash X_i$. Choose a $V(\mathcal{Q}'_i)$-good web $Z_{t}$ with $t\in[2|\mathbf{S}|]\backslash (I_i\cup I'_{i})$ and let $f(x)=t$.
\rm \item\label{dag2} Find internally vertex-disjoint paths $Q_{x,y}$ for every neighbor $y$ of $x$ in $X_i\cup \mathbf{L}\cup \mathbf{M}$ satisfying \ref{3gopathlab2}-\ref{3gopathlab3}. Once this is done, we add these paths to $\mathcal{Q}_i$ to get $\mathcal{Q}_{i+1}$.
\rm \item\label{dag3} Update bad webs by setting $I'_{i+1}=\{i'\in[2|\mathbf{S}|]: Z_{i'} \ \text{is} \ \text{not} \ V(\mathcal{Q}'_{i+1})\text{-} \text{good}\}$, and define
 $I_{i+1}=I_i\cup\{t\}\backslash I'_{i+1}$, $X_{i+1}=f^{-1}(I_{i+1})$ and replace $f$ with its restriction $f|_{X_{i+1}}$.
\rm  \item\label{dag5} Proceed to \textbf{Step} $(i+1)$ with a good path system $(X_{i+1},I_{i+1},I'_{i+1},\mathcal{Q}_{i+1},$ $f|_{X_{i+1}})$.
\end{enumerate}
Now we give the following claim and postpone its proof temporarily.
\begin{claim}\label{dagcl}
In each step the desired paths in \emph{\ref{dag2}} can be successfully found.
\end{claim}

By Claim \ref{dagcl}, $|I_i\cup I'_i|$ is strictly increasing at each step and the above process must terminate in at most $2|\mathbf{S}|$ steps. Let $(X,I,I',\mathcal{Q},f)$ be the final good path system returned from the above process and $\mathcal{Q}'$ be given as in \ref{3gopathlab3}. Note that the sequence $|X_1|,|X_2|,\ldots$ might not be an increasing sequences, as we may delete some elements when updating the list of bad webs in each step. Next we show that the process must terminate with $X=\mathbf{S}$.

Note that for each $v\in \mathbf{L}\cup\mathbf{M}$, $Z_v$ is $V(\mathcal{Q}')$-good, and $\mathcal{Q}'$ might contain some paths whose vertex set intersects $\mathsf{Int}(Z_{i'})\backslash \mathsf{Ctr}(Z_{i'})$ with $i'\in I'$. As at most $m^2$ paths are added at each step, we have $|I'|\leq \frac{2|\mathbf{S}|m^2\cdot 13m}{11m^{12}}=\frac{26|\mathbf{S}|}{11m^{9}}<|\mathbf{S}|$. Thus, $|I\cup I'|< 2|\mathbf{S}|$, and then the process terminates with $X=\mathbf{S}$. To complete the proof, it remains to show that all connections in \ref{dag2} can be guaranteed in each step.

\begin{proof}[Proof of Claim \ref{dagcl}]\renewcommand*{\qedsymbol}{$\blacksquare$}
Given a good path system $(X_i,I_i,I_i',\mathcal{Q}_i,f|_{X_i})$ and $x\in \mathbf{S}\backslash X_i$, $Z_{f(x)}$ as in \ref{dag1}, we let $\{y_1,\ldots,y_s\}=N_{H_1}(x)\cap (X_i\cup \mathbf{L}\cup\mathbf{M})$. For $j\in f(X_i\cup\{x\})$, we know that $Z_j$ is $V(\mathcal{Q}_i)$-good as $(X_i,I_i,I'_i,\mathcal{Q}_i,f)$ is a good path system. By \ref{3gopathlab3}, $V(\mathcal{Q}_i)$ is disjoint from $W$. Denote by $Z_{|\mathbf{M}|+1}, \ldots, Z_{|\mathbf{M}|+2|\mathbf{S}|}$ the $2|\mathbf{S}|$ webs we found as above. Note that there are at most $m^2$ paths in $\mathsf{Ctr}(Z_j)$ $(j\in[|\mathbf{M}|+1,|\mathbf{M}|+2|\mathbf{S}|])$ are involved in precious connections. Hence, there are at least $(22m^2-m^2)m^{10}-11m^{12}=10m^{12}$ available paths in $\mathsf{Int}(Z_j)\backslash \mathsf{Ctr}(Z_j)$, and their corresponding paths in $\mathsf{Ctr}(Z_j)$ are disjoint from $V(\mathcal{Q}_i)$. Let $U_j\subseteq \mathsf{Ext}(Z_j)$ be the union of the leaves of the stars corresponding to these available paths. Then $|U_j|\geq dm^{12}$. Hence, $U'_j=dm^{12}$ for each $j\in f(X_i\cup\{x\})$. However, for each $j\in[s]$, we have $|W\cup \mathsf{Int}(Z_{f(x)})\cup \mathsf{Int}(Z_{f(y_j)})\cup V(\mathcal{Q}_i)|\leq 30dm^{11}+2\cdot22m^{12}+15m|\mathcal{Q}|\leq \frac{\rho(dm^{12})dm^{12}}{4}$ (if $Z_k$ is a web with $k\in[|\mathbf{M}|]$, then $\mathsf{Int}(Z_k)=0$ in this inequality). Similarly, the case when $Z_{f(y)}$ is a $(c_0d,m^{13},2m)$-unit, also witnesses such a vertex set $U_y\subseteq \mathsf{Ext}(Z_{f(y)})$ of size at least $dm^{12}$.  Thus, we can find the desired path $Q_{f(x),f(y_j)}$ connecting $Z_{f(x)}$ and $Z_{f(y_j)}$ while avoiding $W\cup \mathsf{Int}(Z_{f(x)})\cup \mathsf{Int}(Z_{f(y_j)})\cup V(\mathcal{Q}_i)$.
\end{proof}

\noindent
\textbf{Second round: Finding the desired paths (in $G$) for the adjacencies in $H_2$.}

Let $\mathcal{Q}$ be the resulting family of paths for the adjacencies in $H_1$ and $f$ be the resulting embedding of $V(H)$ returned from the first round. Note that $|V(\mathcal{Q})|\leq 13m\cdot e(H_1)\leq13m\cdot e(H)<13dm^2$. As the arguments as above, $V(\mathcal{Q})\cap (\bigcup_{v\in\mathbf{L}\cup\mathbf{M}}\mathsf{Int}(Z_{v}))=\varnothing$. Further,
$|\mathsf{Ext}(Z_v)|\geq dm^{12}$ for each $v\in\mathbf{L}\cup\mathbf{M}$. Let $W^*=V(\mathcal{Q})\cup (\bigcup_{v\in\mathbf{L}\cup\mathbf{M}}\mathsf{Int}(Z_v))$, then $|W^*|\leq 30dm^{10}$.

Let $I\subseteq E(H_2)$ be a maximum set of edges for which there exists a collection $\mathcal{P}_I=\{P_e:e\in I\}$ of internally vertex-disjoint paths under the following rules.
\stepcounter{propcounter}
\begin{enumerate}[label = ({\bfseries \Alph{propcounter}\arabic{enumi}})]
\rm\item\label{3pathlab1} For each $xy=e\in E(H_2)$, $P_e$ is a $z_{f(x)},z_{f(y)}$-path of length at most $13m$ and $P_e$ is disjoint from $W^*\backslash(\mathsf{Int}(Z_{f(x)})\cup\mathsf{Int}(Z_{f(y)}))$.
\rm\item\label{3pathlab2} $P_e$ begins (or ends) with the unique subpath within the web $Z_{f(x)}$ (resp. $Z_{f(y)}$) connecting the core vertex $z_{f(x)}$ (resp. $z_{f(y)}$) and some vertex in $\mathsf{Ext}(Z_{f(x)})$.
\end{enumerate}
Observe that every $v\in \mathbf{L}$ witnesses at least $c_0d-d_H(v)$ available branches in the unit $Z_{f(v)}$ and every $v\in \mathbf{M}$ witnesses at least $22d_H(v)-d_H(v)$ branches in $\mathsf{Ctr}(Z_{f(v)})$, which are internally disjoint from $V(\mathcal{Q}\cup \mathcal{P}_I)$. For every $x\in \mathbf{L}\cup\mathbf{M}$, let $V_x\subseteq \mathsf{Ext}(Z_x)$ be the union of the leaves from the pendant stars attached to one end of these available paths. Then $|V_x|\geq \min\{(c_0d-d_H(v))m^{13},\frac{21dm^{12}}{10}\}\geq dm^{12}$.
\begin{claim}\label{lastcl}
$I=E(H_2)$.
\end{claim}
\begin{proof}[Proof of Claim \ref{lastcl}]\renewcommand*{\qedsymbol}{$\blacksquare$}
Suppose to the contrary that there exists an edge $e=xy\in E(H_2)\backslash I$ with no  paths in $\mathcal{P}_I$ between their corresponding webs, say $Z_{f(x)}, Z_{f(y)}$. Then $|\mathcal{P}_I|\leq 13dm^2$, and $|W^*|+|\mathcal{P}_I|\leq 30dm^{10}+13dm^2<32dm^{10}<\frac{1}{4}\rho(dm^{12})dm^{12}$.
By Lemma \ref{distance}, there is a path $P_e'$ of length at most $m$ between $\mathsf{Ext}(Z_{f(x)})$ and $\mathsf{Ext}(Z_{f(y)})$ while avoiding $W^*\cup V(\mathcal{P}_I)$, and we can easily extend $P_e'$ into a path $P_e$ of length at most $13m$ connecting $z_{f(x)}$ and $z_{f(y)}$. Hence, $\{P_e\}\cup \mathcal{P}_I$ yields a contradiction to the maximality of $\mathcal{P}_I$.
\end{proof}
In conclusion, the resulting families of paths $\mathcal{Q}$ in the first round and $\mathcal{P}_I$ in the second round form a copy of $TH$ as desired.
\end{proof}

\end{appendix}

\end{document}